\documentclass[12pt,reqno,oneside]{amsart}
\usepackage{mathdots}
\usepackage{amsmath}
\usepackage{amssymb}
\usepackage{pgfplots}
\pgfplotsset{compat=1.18}
\usepackage{amsthm}
\usepackage{graphics}
\usepackage{eucal}
\usepackage{bm}
\usepackage[colorlinks]{hyperref}
\usepackage[capitalize,nameinlink,noabbrev]{cleveref}
\usepackage{mathrsfs}
\usepackage[normalem]{ulem}
\usepackage{color}
\usepackage[dvipsnames]{xcolor}
\usepackage{mathtools}
\usepackage{geometry}
\newtheorem{theoremKM}{Theorem KM}

\newtheorem{theorem}{Theorem}[section]
\newtheorem{proposition}{Proposition}[section]
\newtheorem{lemma}{Lemma}[section]
\newtheorem{corollary}{Corollary}[section]

\theoremstyle{remark}
\newtheorem{remark}{Remark}[section]
\theoremstyle{definition}
\newtheorem{definition}{Definition}[section]

\numberwithin{equation}{section}

\newcommand{\vv}[1]{{\mathbf{#1}}}
\newcommand{\ep}{\epsilon}

\newcommand{\be}{\mathbf{e}}

\newcommand{\diag}{\operatorname{diag}}

\newcommand{\bx}{\mathbf{x}}
\newcommand{\bU}{\mathbf{U}}
\newcommand{\bT}{\mathbf{T}}

\newcommand{\Leb}{\mathcal{L}}
\newcommand{\bp}{\mathbf{p}}

\newcommand{\bv}{\mathbf{v}}
\newcommand{\bb}{\mathbf{b}}
\newcommand{\bu}{\mathbf{u}}
\newcommand{\bw}{\mathbf{w}}
\newcommand{\bs}{\mathbf{s}}
\newcommand{\f}{\mathbf{f}}

\newcommand{\ve}{\varepsilon}

\newcommand{\bz}{\mathbf{z}}

\newcommand{\ba}{\mathbf{a}}

\newcommand{\Span}{\operatorname{Span}}

\newcommand{\Deltanew}{\Delta_{new}}

\newcommand{\by}{\mathbf{y}}

\newcommand{\Q}{\mathbb{Q}}

\newcommand{\Z}{\mathbb{Z}}
\newcommand{\R}{\mathbb{R}}
\newcommand{\N}{\mathbb{N}}
\newcommand{\fff}{\mathit{f}}

\newcommand{\cM}{\mathcal{M}}

\newcommand{\inv}{^{\text{-}1}}

\usepackage{mathtools}
\usepackage{tikz-cd}
\usepackage{graphicx}
\begin{document}

\title[Rational points and manifolds]{How close can rational points get to a manifold?}

\author{Victor Beresnevich}
\address{\textbf{Victor Beresnevich} \\
Department of Mathematics, University of York, UK }
\email{victor.beresnevich@york.ac.uk}
\author{Shreyasi Datta}
\address{\textbf{Shreyasi Datta} \\
TSMU, ISI Bangalore, India}
\email{shreyasi@isibang.ac.in}

\date{}

\maketitle
{\centering \textit{In memory of Maurice Dodson}\par}

\begin{abstract}
We prove the heuristically predicted lower bound for the number of rational points of height at most $Q$ lying within $\ep/Q$ of a fixed analytic nondegenerate manifold in $\mathbb{R}^n$, provided that $\ep \asymp Q^{-\tau}$ for some $\tau \leq 3/(2m+1)$, where $m$ is the codimension of the manifold. Our result establishes the lower bound well beyond the previously conjectured range $\tau \leq 1/m$, and improves upon a recent result of Schindler, Srivastava, and Technau, who established the same lower bound for $\tau \leq 3/(2n-1)$.
\end{abstract}

\section{Introduction}

The main goal of this paper is to understand how close rational points of bounded height can get to a generic submanifold of $\mathbb{R}^n$. Here we address this challenging problem for the natural class of analytic nondegenerate manifolds, which includes all connected analytic submanifolds of $\R^n$ not contained in any proper affine subspace \cite{KM}.

More precisely, we investigate the following counting function:
\begin{equation}
    N(\ep, t):=\#\left\{(q, \bp)\in \Z_{\rm prim}^{n+1}~|~\inf_{\by\in \mathcal{M}}\left\Vert\frac{\bp}{q}-\by\right\Vert\le\frac{\ep}{e^t},~ e^{t-1}\leq q< e^t\right\},
\end{equation}
where $\cM$ denotes a fixed bounded submanifold of $\R^n$, $t\in\N$, $0<\ep<1$ and $\Z^{n+1}_{\rm prim}$ is the subset of $\Z^{n+1}$ of primitive points. Recall that $(a_0,\dots,a_n)\in\Z^{n+1}$ is primitive if and only if $\gcd(a_0,\dots,a_n)=1$, and so primitive points $(q,\bp)\in\Z^{n+1}$ with $q>0$ are in one to one correspondence with rational points $\bp/q\in\Q^n$. If needed, we will also write $N_{\cM}(\ep,t)$ for $N(\ep,t)$ to explicitly show the dependence on $\cM$.

In what follows, $d$ and $m$ will denote the dimension and codimension of $\cM$ respectively. Thus, $d+m=n$. Given two quantities $X$ and $Y$, we write $X\ll Y$ if there exists a constant $C>0$ such that $X\leq CY$. We write $X\asymp Y$ if $X\ll Y$ and $Y\ll X$.

The question alluded to in the title can now be rephrased as follows: how small can $\ep$ be, as a function of $t$, while satisfying $N(\ep,t)>0$? It is also natural to ask about the asymptotic behaviour of $N(\ep,t)$ as $t\to\infty$ and $\ep\to0$ \cite{Ber12, Huang2020}. Assuming that the manifold $\cM$ neither aligns with nor repels rational points in $\R^n$ in some pathological way, and that $\ep$ is not too small, it is reasonable to expect the following volume-based 
$$
\textbf{Heuristic estimate:}\qquad N(\ep,t)\asymp \ep^m e^{t(d+1)}.\phantom{\textbf{Heuristic estimate:}\qquad}
$$

Motivated by its natural applications to Diophantine approximation, the problem of determining the asymptotic behaviour of $N(\ep,t)$ was introduced by Bernik and Dodson in their celebrated monograph \cite[\S4.4, p.~95]{BernikDodsonBook}. Their primary motivation was to obtain an analogue of the Jarník--Besicovitch theorem, see \cite{MR5011579} and references within, for the unit circle $x^2+y^2=1$ and, more generally, for planar curves.

For $C^3$ nondegenerate planar curves the problem was solved in \cite{BDV07, Huxley, VV06}, with the $C^3$ condition removed in \cite{BZ10, Hua15}. In particular, 
\begin{equation}\label{LBdim2}
N(\ep,t)\gg \ep e^{2t}\qquad\text{whenever $\ep\gg e^{-t}$,}
\end{equation}
while the complementary upper bound holds, assuming the curvature is uniformly bounded between positive constants, whenever $\ep\gg e^{-t}t^p$ for some sufficiently large $p>0$. Furthermore, the condition $\ep\gg e^{-t}$ in \eqref{LBdim2} is sharp within the class of nondegenerate planar curves. For example, for the circle $x^2+y^2=3$ there exists an explicitly computable constant $c_0>0$ such that $N(\ep,t)=0$ whenever $\ep<c_0e^{-t}$.

Of course, the above example is rather special, since the curve is a rational quadric curve. By contrast, Oppenheim's conjecture, famously proved by Margulis in 1986, implies \cite[\S5.1]{MR4422058} that if $\cM$ is a quadratic planar curve not defined over $\Q$, then for every $c>0$ one has $N(\ep,t)>0$ with $\ep=ce^{-t}$ for arbitrarily large values of $t$.

More recently, Smith \cite{Sch2025} extended this result to sufficiently smooth nondegenerate curves in $\R^2$ and, similarly, to hypersurfaces in $\R^n$ that are not quadratic hypersurfaces defined over $\Q$. Furthermore, for analytic nondegenerate hypersurfaces, that is, when $m=1$ and $d+1=n$, the answer to the question posed in the title is essentially the same as for planar curves, which are precisely the hypersurfaces in $\R^2$, see \cite{Ber12}. Furthermore, 
\[
N(\ep,t)\gg \ep e^{nt}\qquad\text{whenever $\ep\gg e^{-t}$,}
\]
and the lower bound on $\ep$ is sharp \cite[Example 1.2]{Ber12}.

For manifolds of codimension $m>1$, the situation is more complicated. For convenience, in what follows, we will write $\eta(\cM)\ge \eta_0$ to mean that
\begin{equation}
N(\ep,t)\gg \ep^m e^{t(d+1)}
\quad\text{for all } t\gg1 \text{ and } 1>\ep\gg e^{-t\eta_0},
\end{equation}
where the implied constants may depend on  $n$, $d$, $\eta_0$, and $\cM$.

As shown in \cite{Ber12}, for any analytic nondegenerate manifold $\cM$ of codimension $m$ in $\R^n$
\begin{equation}\label{eq: lb Ber}
\eta(\cM)\geq \frac{1}{m}.
\end{equation}
Furthermore, for nondegenerate curves (i.e. $d=1$), it was shown in \cite{Ber12} and \cite{BVVZ2018} that
\begin{equation}\label{eq: lb SST}
\eta(\cM)\geq \frac{3}{2n-1},
\end{equation}
which, for $n>2$, is stronger than the general bound $\frac{1}{m}$ given by \eqref{eq: lb Ber}.

Extending \eqref{eq: lb Ber} beyond the analytic setting turned out to be highly non-trivial. The best known bound \cite{BD} valid for all nondegenerate manifolds of dimension $d>1$ is
\[
\eta(\cM)\geq \frac{1}{n-1}.
\]
Additional smoothness assumptions were required to obtain stronger estimates. In particular, Schindler--Srivastava--Technau \cite{SST} recently proved that, for any sufficiently smooth nondegenerate manifold $\cM$, for any $\delta>0$
\begin{equation}\label{SST0}
\eta(\cM)\geq \frac{3}{2n-1}-\delta.
\end{equation}
Apart form $\delta$, for nondegenerate curves this is the same as \eqref{eq: lb SST}; see also Remark~\ref{rem: count for more general} below showing how $\delta$ can be removed. Within the class of analytic nondegenerate manifolds, \eqref{SST0} improves upon \eqref{eq: lb Ber} whenever $m>(2n-1)/3$. In this paper we improve both \eqref{eq: lb Ber} and \eqref{SST0} for any analytic manifolds with $m>1$. Our main result reads as follows.

\begin{theorem}\label{thm: main}
Let $\mathcal{M}$ be an analytic nondegenerate manifold in $\R^n$ of dimension $d$ and codimension $m$. Then for all sufficiently large $t$ and $\ep\in(0,1)$,
\begin{equation}\label{eq: main thm}
    N(\ep, t)\gg \ep^m e^{t(d+1)}\qquad\text{if }\ep\gg e^{-\frac{3t}{2m+1}},
\end{equation}
where the implied constants depend on $\mathcal{M}$. In particular,
\begin{equation}\label{OURS}
\eta(\mathcal{M})\geq \frac{3}{2m+1}.
\end{equation}
\end{theorem}

Note that $\frac{3}{2m+1}>\max\left\{\frac{1}{m}, \frac{3}{2n-1}\right\}$ for all $m>1$, with equality at $m=1$ and $m=n-1$. Thus, our result improves both \eqref{eq: lb Ber} and \eqref{SST0}. Also, it is interesting to note that the above new exponent (for all nondegenerate manifolds) was not suggested or explored before in any previous work.

The exponent in \eqref{OURS} is achieved by improving the measure estimate established by Bernik--Kleinbock--Margulis \cite{BKM} which we shall refer to as the \emph{BKM estimate}. We note that the BKM estimate was used in all previous works of this ilk, including \cite{BD}, \cite{BY}, and \cite{SST}. Our paper's novelty is in improving the existing BKM estimate so that it is useful when certain \textit{derivatives} are quite small. The improved BKM estimate, stated as Theorem~\ref{new thm: BKM} below, needs the analyticity assumption and is derived from the quantitative nondivergence estimate of Kleinbock--Margulis \cite{KM}, stated in \S\ref{sec: KM}. In short, the main ingredients behind the proof of \eqref{OURS} are the improved BKM estimate (Theorem~\ref{new thm: BKM}) together with a new duality argument for detecting rational points near manifolds, inspired by \cite{BY}, where duality was used to obtain upper bounds.

\subsection{An improvement of the BKM estimate}

Here, we state and discuss our improvement (Theorem~\ref{new thm: BKM}) of the BKM estimate  \cite{BKM}, which concerns the set
\begin{equation}\label{new set}
\Deltanew(\delta, K, T):=\left\{\bx\in \bU:\exists\;(a_0,\ba)\in\Z\times\Z^n_{\neq\bf0}\;\;\left.
\begin{array}{l}
|a_0+F(\bx)\ba^\top|<\delta\\[1ex]
\|\nabla F(\bx)\ba^\top\|<K\\[1ex]
\vert a_i\vert<T ~~(d+1\leq i\leq n) 
\end{array}
\right.\right\}.
\end{equation}
Here, $\ba=(a_1,\dots,a_n)$ and $\delta,K,T$ are positive real parameters, $\nabla$ denotes the gradient of a real valued function, $\|\cdot\|$ denotes the supremum norm, $\ba^\top$ is the transpose of $\ba$, and
$F:\bU\to\R^n$ is a map defined on an open set $\bU\subset\R^d$ of the form
\begin{equation}\label{eq: F}
F(\bx) = (\bx, f(\bx)),\qquad\text{where $\bx \in \bU$, 
$f(\bx) = (f_i(\bx))_{i=1}^m$, and $m = n - d$.}
\end{equation}
In what follows, $\mu_d$ denotes $d$ dimensional Lebesgue measure.

\begin{theorem}\label{new thm: BKM}
Let $\bU\subset \R^d$ an open set, $m+d=n$, and the map $F:\bU\to\R^d$ of the form \eqref{eq: F} be analytic. Suppose further that $F$ is nondegenerate at $\bx_0$ for some $\bx_0\in \bU$. Then there exists an open neighborhood $V\ni\bx_0$ and $\alpha>0$ with the following property. For any ball $B\subset V$, there exists a constant $E$ such that for any choice of $\delta>0$, $K>0$, and $T>0$ satisfying
\begin{equation}\label{condition on K T}
 \max\{\delta^{m+2}, K^{m+2}\}\leq \delta K T^m,\qquad \delta K\leq T^{2},
\end{equation}
we have that
\begin{equation}\label{eq: crucial improvement}
\mu_d(\Deltanew(\delta, K, T)\cap B)\leq E (\delta K T^m)^{\alpha/(m+2)}\mu_d(B)\,.
\end{equation}
\end{theorem}

\begin{remark}
Observe that, for \eqref{eq: crucial improvement} to be non-trivial, it is necessary that
$\delta K T^m<1$. Hence, in view of \eqref{condition on K T}, any meaningful application of Theorem~\ref{new thm: BKM} requires that $\delta<1$ and $K<1$. Furthermore, in any interesting application we must have $T\ge 1$; otherwise, in view of the form of $F$, the set $\Deltanew(\delta,K,T)$ is empty. 
\end{remark}

\begin{remark}
The BKM estimate \cite[Theorem 1.4]{BKM} implies that 
\begin{equation}\label{BKMimplication}
\mu_d(\Deltanew(\delta, K, T)\cap B)\ll\left(\delta K T^{n-1}\right)^{\alpha} \mu_{d}(B) 
\end{equation}
provided that $\delta>0$, $K>0$, and $\delta^{n+1}< \delta K T^{n-1}$.
Furthermore, the exponent $\alpha$ in the BKM estimate is given explicitly in terms of the degree of nondegeneracy of the manifold, as well as its dimension and codimension. The estimate \eqref{BKMimplication} was subsequently strengthened in \cite{BD}, where it was shown that the measure of interest is bounded by
\begin{equation}\label{BKM_BD}
\delta K^d T^m\,\mu_d(B)
+
\left(\delta\min\{K,r(B)^{-1}\}T^{n-1}\right)^\alpha,
\end{equation}
although in this case the exponent $\alpha$ is no longer explicit. It turns out that neither \eqref{BKMimplication} nor \eqref{BKM_BD} is useful for detecting rational points of denominator $\asymp e^t$ whose distance from the manifold is significantly smaller than $\ep e^{-t}$ for 
$\ep= e^{-3t/(2n-1)}$, let alone $\ep= e^{-t/m}$.
\end{remark}

\section{Preliminaries}

\subsection{Notation, definitions and further assumptions}\label{sec: notation}

Throughout, points $\bx\in\R^k$ will be written as $(x_1,\dots,x_k)$, and
$\bx\cdot\by=\sum_{i=1}^k x_i y_i$ will stand for the inner product of $\bx,\by\in\R^k$. We will use $\|\cdot\|$ and $\|\cdot\|_2$ to denote the  supremum and Euclidean norms on $\R^k$ respectively.

Given two quantities $X$ and $Y$, we write $X=\pm Y$ if $X=Y$ or $X=-Y$.

Following \cite{BY}, for any matrix
$A\in \mathrm{GL}_{n+1}(\R)$, we define its dual matrix by
\[
A^\star:=\sigma(A^\top)^{-1}\sigma,
\]
where $\sigma$ is the long Weyl element, namely
\[
\sigma=
\begin{bmatrix}
0 & \dots & 1 \\
\vdots & \iddots & \vdots \\
1 & \dots & 0
\end{bmatrix}
\in \mathrm{GL}_{n+1}(\R).
\]

Given a lattice $\Lambda\subset\R^{n+1}$, we denote by
$\lambda_i(\Lambda)$ the $i$th successive minimum of $\Lambda$
with respect to the unit ball in $\R^{n+1}$.

Let $F:\bU\to\R^n$ be a map defined on an open subset $\bU\subset\R^d$, $\bx_0\in\bU$. Recall from \cite{KM} that $F$ is called non-degenerate at $\bx_0$ if for some $l\in\N$, $F\in C^l$ in a neighbourhood of $\bx_0$ and 
\begin{equation}\label{nondegen}
\operatorname{rank}\{\partial_\beta F(\bx_0)~\vert~ 1\le|\beta|\le l\}=n.
\end{equation}
Here $\beta=(\beta_1,\dots,\beta_d)$ is a multi-index of non-negative integers, $|\beta|=\beta_1+\dots+\beta_d$ and $\partial_\beta F:=\partial^{|\beta|}F/\partial x_1^{\beta_1}\dots\partial x_d^{\beta_d}$ is the corresponding partial derivative. It is readily seen that if $F$ is additionally analytic then $F$ is nondegenerate at $\bx_0$ if and only if for any sufficiently small ball $B_0$ centred at $\bx_0$ the functions 
$$
1, F_1,\dots, F_n
$$ 
restricted to $B_0$ are linearly independent over $\R$. The nondegeneracy of a manifold $\cM\subset\R^n$ at any point $\by_0\in\cM$ is defined via the nondegeneracy of its parameterisation $F$ of a suitable neighbourhood of $\by_0$. Note that for any connected analytic manifold $\mathcal{M}\subset\R^n$, if it is nondegenerate at one point, then it is nondegenerate at every point. In this case we will say that the manifold is nondegenerate. A connected analytic manifold is nondegenerate if and only if it is not contained in a proper affine subspace of $\R^n$ \cite{KM}.

Since we are interested in obtaining a lower bound for the counting function, it suffices to restrict our attention to a subset of a given manifold. By the Implicit Function Theorem, we may therefore assume without loss of generality that $\cM=F(\bU)$ for some 
map $F:\bU \to \R^n$ of the form \eqref{eq: F} defined on an open subset $\bU \subset \R^d$.
For the rest of the paper, unless otherwise stated, $F$ will be assumed to be analytic.

It will be convenient to introduce the extended map
$\tilde F:\bU\to\R^{n+1}$ defined by
\begin{equation}\label{Ftilde}
\tilde F(\bx):=(1,F(\bx))=(1,\bx,f(\bx)).
\end{equation}

We also assume, without loss of generality, that there exists $M>0$ such that 
\begin{equation}\label{eq: bound on derivatives}
   \max_{k}\max_{i,j} \sup_{\bx\in \bU}\vert\partial_{ij}f_{k}(\bx)\vert\leq M.
\end{equation}

Next, with $\bU$ and $f$ as in \eqref{eq: F}, define 
\begin{equation}\label{def_u_1}
  u_1(\bx)=
      \begin{bmatrix}
    1 & 0 & \dots & 0 & -\partial_d f_m(\bx) & \dots &  -\partial_1 f_m(\bx) & f_{m}(\bx)- \sum_{i=1}^d x_i\partial_i f_m(\bx)\\
    \vdots & \vdots & \vdots & \vdots & \vdots & \vdots & \vdots & \vdots \\
    0 & 0 & \dots & 1 & -\partial_d f_1(\bx) & \dots &  -\partial_1 f_1(\bx) & f_{1}(\bx)- \sum_{i=1}^d x_i\partial_i f_1(\bx)\\
    0 & 0 & \dots & 0 & 1 & 0 & \dots & x_{d}\\
     \vdots & \vdots & \vdots & \vdots & \vdots & \vdots & \vdots & \vdots \\
     0 & 0 & \dots & 0 & 0 &  \dots & 1 &  x_{1}\\
     0 & 0 & \dots & 0 & 0 &  \dots & 0 &  1    
\end{bmatrix}.\end{equation}

A straightforward calculation \cite{BY} shows that the dual of $u_1$ is given by  \begin{equation}\label{eqn:dual u_1(x)}
        u_1^\star(\bx)=\begin{bmatrix}
            1 & -\bx & -f(\bx)\\
            0 & \mathrm{I}_d & J(\bx)\\
            0 & 0   & \mathrm{I}_m
        \end{bmatrix}\,,
    \end{equation} 
where 
$$
J(\bx)=\big[\partial_i\fff_j(\bx)\big]_{1\leq i\leq d,\, 1\leq j \leq m}\,, $$ 
is the Jacobian of the map $\bx\mapsto f(\bx)=(f_1(\bx),\dots,f_m(\bx))$.

Finally, following \cite{BKM}, we use the notion $\tilde\nabla$ for the skew gradient.
Recall that, given a pair of functions $g_1,g_2:\R^d\to\R$, the skew gradient
$\tilde\nabla(g_1,g_2):\R^d\to\R^d$ is defined by
\begin{equation}\label{SG}
\tilde\nabla(g_1,g_2)(\bx)
=
g_1(\bx)\nabla g_2(\bx)-g_2(\bx)\nabla g_1(\bx).
\end{equation}

\subsection{Axillary lemmas}

In this subsection, we gather together several well-known statements about lattices and exterior algebra.

We will use the standard (orthonormal) basis of $\bigwedge^r\R^{n+1}$ given by 
\begin{equation}\label{basis_wedge_k}
\be_{I}=\be_{i_1}\wedge\dots\wedge\be_{i_r},\qquad I=\{i_1<\dots<i_r\}\subset\{0,\dots,n\}.
\end{equation}
Here $\be_{0},\dots,\be_n$ is the standard basis of $\R^{n+1}$. Thus, any $\bv\in \bigwedge^r \R^{n+1}$ can be written in the basis \eqref{basis_wedge_k} as follows 
$$
\bv=\sum v_I \be_{I},
$$
where $v_{I}\in \R$. The supremum and Euclidean norms of $\bv$ are given by
$$
\|\bv\|=\sup|v_I|\quad\text{and}\quad \|\bv\|_2=\Big(\sum_I|v_I|^2\Big)^{1/2}
$$
respectively.

Now, we recall some widely known relations between linear subspaces of a finite dimensional Euclidean space $E$ and multi-vectors, see \cite[\S3.3]{Ber12} for a summary.
The vectors $\bv_1,\dots,\bv_r$ are linearly independent if and only if $\bv_1\wedge\dots\wedge\bv_r\neq\boldsymbol{0}$. Also, $\bv_1,\dots,\bv_r$ is a basis of a linear subspace $V\subset E$ of dimension $r$
if and only if
\begin{equation}\label{lin_sub_wedge}
V=\{\bv\in E~\vert~\bv\wedge\bv_1\wedge\dots\wedge\bv_r=\boldsymbol{0}\}.
\end{equation}
Furthermore, $\bv'_1,\dots,\bv'_r$ is any other basis of $V$ if and only if $\bv'_1\wedge\dots\wedge\bv'_r=c\bv_1\wedge\dots\wedge\bv_r$ for some non-zero $c\in\R$. 

Now we recall some definitions and statements about lattices. By \eqref{lin_sub_wedge},
for any linearly independent collection $\bv_1,\dots,\bv_r\in E$ and for any subset 
$\Lambda\subseteq E$
\begin{equation}\label{eq2.11}
\Span_\R(\bv_1,\dots,\bv_r)\cap\Lambda=\{\bv\in\Lambda~\vert~\bv\wedge\bv_1\wedge\dots\wedge\bv_r=\boldsymbol{0}\}.
\end{equation}
In the case $\Lambda$ is a lattice in $E$, \eqref{eq2.11} is clearly a discrete subgroup of $E$. Furthermore, if $\bv_1,\dots,\bv_r\in\Lambda$ are linearly independent, then the rank of \eqref{eq2.11} is equal to $r$.

The collection $\bv_1,\dots,\bv_r$ of vectors in a lattice $\Lambda$  is called {\em primitive}, see \cite[Definition~2 on page~21]{MR893813} or \cite{KM}, if $\bv_1,\dots,\bv_r$ is the $\Z$-basis of the discrete subgroup \eqref{eq2.11}, that is, the vectors
$\bv_1,\dots,\bv_r$ are linearly independent and
$$
\Span_\Z(\bv_1,\dots,\bv_r)=\Span_\R(\bv_1,\dots,\bv_r)\cap\Lambda.
$$

The following lemma gives two equivalent characterisations of primitive collections.

\begin{lemma}\label{primitive}
Let $\Lambda$ be a lattice in a $d$-dimensional Euclidean space $E$, $\vv b_1,\dots,\vv b_{d}$ be any basis of $\Lambda$ and $\bv_1,\dots,\bv_r\in\Lambda$. Then the following three statements are equivalent\/{\rm:}
\begin{itemize}
\item[{\rm(i)}] $\bv_1,\dots,\bv_r$ is a primitive collection of $\Lambda$\/{\em;}
\item[{\rm(ii)}] $\bv_1,\dots,\bv_r$ can be completed to a basis of $\Lambda$\/{\em;}
\item[{\rm(iii)}] the greatest common divisor of the $r\times r$ determinants 
$$
\det(v_{i,j_\ell})_{1\le i,\ell\le r}\qquad \{j_1<\dots<j_r\}\subset\{0,\dots,n\}
$$
is $1$, where $(v_{i,1},\dots,v_{i,d})$ are the coordinates of $\bv_i$ in the basis $\vv b_1,\dots,\vv b_d$.
\end{itemize}
\end{lemma}

The equivalence of (i) and (ii) can be found in \cite[Theorem~5 and the comment prior to  Theorem~5 on page~21]{MR893813}. The equivalence of (ii) and (iii) can be found in \cite[Lemma~2 in Chapter~1 on page~15]{MR1434478}.

We now give several statements for the case when the lattice $\Lambda=\Z^{n+1}$ in the Euclidean space $E=\R^{n+1}$. First of all, we have the following simple characterisation of primitive collections in terms of the exterior product which follows from  Lemma~\ref{primitive}(iii).

\begin{corollary}\label{primitive2}
A collection $\bv_1,\dots,\bv_r\in\Z^{n+1}$ is primitive if and only if the greatest common divisor of the coordinates of $\bv_1\wedge\dots\wedge\bv_r$ in the standard basis \eqref{basis_wedge_k} is $1$. 
\end{corollary}

The following statement is a generalisation of \cite[Lemma 5.21]{BNY}.

\begin{lemma}\label{lem: dual prim}
Let $1\leq r\leq n$, and $\bv_1,\dots,\bv_r\in \Z^{n+1}$ be a linearly independent collection. Then, there exists a linearly independent collection $\bu_1,\dots, \bu_{n+1-r}\in \Z^{n+1}$ such that
\begin{equation}\label{dot}
    \bv_i\cdot \bu_{j}=0, 1\leq i\leq r, 1\leq j\leq n+1-r,
    \end{equation}
    and 
\begin{equation}\label{same norm}
        \Vert \bv_1\wedge\dots\wedge\bv_r\Vert_2=\Vert \bu_1\wedge\dots\wedge\bu_{n+1-r}\Vert_2.
    \end{equation}
\end{lemma}

\begin{proof}
For primitive collections $\bv_1,\dots,\bv_r$ this lemma can be found in \cite[Lemma 5.21]{BNY} which in turn follows from  \cite[Lemma~5G]{MR1176315}. Now, fix any collection $\bv_1,\dots,\bv_r\in \Z^{n+1}$ of linearly independent vectors. The latter is equivalent to
$\bv_1\wedge\dots\wedge\bv_r\neq \mathbf{0}$. Let $\bv_1',\dots,\bv_r'$ be any $\Z$-basis of $\Span_\R(\bv_1,\dots,\bv_r)\cap\Z^{n+1}$. Then, $\bv_1',\dots,\bv_r'$ is a primitive collection and, by the primitive case of this lemma, there exists a linearly independent collection $\bu'_1,\dots, \bu'_{n+1-r}\in \Z^{n+1}$ such that \eqref{dot} and \eqref{same norm} hold for vectors $\bv'_i$ and $\bu'_j$. Since $\bv'_1,\dots,\bv'_r$ is also a basis of $\Span_\R(\bv_1,\dots,\bv_r)$, we have that $\bv_1\wedge\dots\wedge\bv_r=c\bv'_1\wedge\dots\wedge\bv'_r$ for some non-zero $c\in\R$. Since both multi-vectors have integer coordinates in the standard basis, $c\in\Q$, and since $\bv'_1,\dots,\bv'_r$ is primitive, by Corollary~\ref{primitive2}, the coordinates of $\bv'_1\wedge\dots\wedge\bv'_r$ are coprime integers. Hence, we must have that $c\in\Z$. Let $\bu_1=c\bu'_1$, $\bu_2=\bu'_2,\dots,\bu_{n+1-r}=\bu'_{n+1-r}$. Now, \eqref{dot} and \eqref{same norm} follow immediately.
\end{proof}

\begin{remark}
Equation \eqref{same norm} also holds for the supremum norm, since $\bu_1\wedge\dots\wedge\bu_{n+1-r}$ can be defined by applying the Hodge star operator to $\bv_1\wedge\dots\wedge\bv_r$. In this case, the coordinates of $\bu_1\wedge\dots\wedge\bu_{n+1-r}$ and  $\bv_1\wedge\dots\wedge\bv_r$ in the standard basis are the same up to permutation and sign changes, see \cite[\S3.2]{Ber12}. This way one can obtain an alternative proof of Lemma~\ref{lem: dual prim}.
\end{remark}

The next lemma can be explicitly found in \cite[Lemma 5.20]{BNY}.

\begin{lemma}\label{lem: dual basis consq}
    Let $1\leq r\leq n$, and $\bv_1,\dots,\bv_r\in \R^{n+1}$ and $\bu_1,\dots,\bu_{n+1-r}\in \R^{n+1}$ be linearly independent collections satisfying \eqref{dot} and \eqref{same norm}. Then for any $\bz_1,\dots,\bz_r\in \R^{n+1}$, 
    we get that 
$$
\vert (\bz_1\wedge\dots\wedge\bz_r)\cdot (\bv_1\wedge\dots\wedge\bv_r)\vert =\Vert \bz_1\wedge \dots\wedge \bz_r\wedge \bu_{1}\wedge\dots\wedge\bu_{n+1-r}\Vert.
$$
\end{lemma}

Recall that for any $\bu_1,\dots,\bu_{n+1}\in\R^{n+1}$, the multi-vector $\bu_1\wedge\dots\wedge\bu_{n+1}$ has only one coordinate in the standard basis which consists of only one multi-vector $\be_0\wedge\dots\wedge\be_n$. This coordinate is equal to the  determinant of the matrix composed from the rows of the coordinates of $\bu_1,\dots,\bu_{n+1}$ in the standard basis. Hence, by Lemma~\ref{primitive}, $|\bu_1\wedge\dots\wedge\bu_{n+1}|=1$ if and only if the collection $\bu_1,\dots,\bu_{n+1}\in\Z^{n+1}$ is  primitive (equivalently, a basis of $\Z^{n+1}$).

The final statement of this subsection is an extension of Property~(ii) of Lemma~\ref{primitive} to primitive collections within discrete subgroups.

\begin{lemma}\label{lem: helpful-general}
Let $\bu_1,\dots,\bu_{k+1}\in \Z^{n+1}$ such that 
\begin{equation}\label{indepen}
\bu_1\wedge\dots\wedge\bu_{k+1}\neq\mathbf{0}.
\end{equation}
Let $\ell$ be the greatest common divisor  of the coordinates of 
$\bu_1\wedge\dots\wedge\bu_{k+1}$ in the standard basis \eqref{basis_wedge_k}. Let $\bz_1,\dots,\bz_r\in \Z^{n+1}$ be a primitive collection  such that $1\le r\le k+1$ and
\begin{equation}\label{cond2.14}
\bz_{i}\wedge \bu_1\wedge\dots\wedge\bu_{k+1}=\mathbf{0} \qquad (1\leq i\leq r).
\end{equation}
Then there exist $\bu_1',\dots,\bu_{k+1-r}'\in \Z^{n+1}$ such that $\bz_1,\dots,\bz_r,\bu'_1,\dots,\bu'_{k+1-r}$ is a primitive collection satisfying
\begin{equation}\label{cond2.15}
\bu_1\wedge\dots\wedge\bu_{k+1}=\pm\ell\,\bz_1\wedge\dots\wedge\bz_{r}\wedge\bu_1'\wedge\dots\wedge\bu_{k+1-r}'.
\end{equation}
\end{lemma}

\begin{proof}
Let $E=\Span_\R(\bu_1,\dots,\bu_{k+1})$ and  
$\Lambda=E\cap\Z^{n+1}$. By \eqref{indepen}, $\dim E=k+1$. By the fact that $\bu_1,\dots,\bu_{k+1}\in \Z^{n+1}$, we also have that $\operatorname{rank}(\Lambda)=k+1$. Thus, $\Lambda$ is a lattice in $E$. 

Since $\bz_i\in\Z^{n+1}$, by \eqref{eq2.11} and  \eqref{cond2.14}, $\bz_i\in\Lambda$ for each $i$. Since $\bz_1,\dots,\bz_r$ is a primitive collection of $\Z^{n+1}$ in $\R^{n+1}$ and $\Span_\R(\bz_1,\dots,\bz_r)\subset E$, we have that
\begin{align*}
\Span_\Z(\bz_1,\dots,\bz_r)&=\Span_\R(\bz_1,\dots,\bz_r)\cap\Z^{n+1}\\
&=\Span_\R(\bz_1,\dots,\bz_r)\cap E\cap\Z^{n+1}\\
&=\Span_\R(\bz_1,\dots,\bz_r)\cap\Lambda.
\end{align*}
Hence, $\bz_1,\dots,\bz_r$ is also a primitive collection of the lattice $\Lambda$ in $E$.

By Lemma~\ref{primitive}(ii), the collection $\bz_1,\dots,\bz_r$ can be completed to a $\Z$-basis of $\Lambda$ by some vectors $\bu_1',\dots,\bu_{k+1-r}'\in\Lambda\subset\Z^{n+1}$. It follows that
\begin{align*}
\Span_\Z(\bz_1,\dots,\bz_{r},\bu_1',\dots,\bu_{k+1-r}')&=\Span_\R(\bz_1,\dots,\bz_{r},\bu_1',\dots,\bu_{k+1-r}')\cap\Lambda\\
&=\Span_\R(\bz_1,\dots,\bz_{r},\bu_1',\dots,\bu_{k+1-r}')\cap\Z^{n+1}
\end{align*}
since $\bz_1,\dots,\bz_{r},\bu_1',\dots,\bu_{k+1-r}'$ is also a basis of $E$. Hence, $\bz_1,\dots,\bz_{r},\bu_1',\dots,\bu_{k+1-r}'$ is a primitive collection of $\Z^{n+1}$.

Since $\bu_1,\dots,\bu_{k+1}$ and $\bz_1,\dots,\bz_{r},\bu_1',\dots,\bu_{k+1-r}'$ span the same linear space, we have that
$$
\bu_1\wedge\dots\wedge\bu_{k+1}=c\,\bz_1\wedge\dots\wedge\bz_{r}\wedge\bu_1'\wedge\dots\wedge\bu_{k+1-r}'
$$
for some non-zero $c\in\R$. Since these two multi-vectors have integer coordinates, $c\in\Q$. Since $\bz_1,\dots,\bz_{r},\bu_1',\dots,\bu_{k+1-r}'$ is  primitive, by Corollary~\ref{primitive2}, $c=\pm \ell$, where $\ell$ is the greatest common divisor  of the coordinates of 
$\bu_1\wedge\dots\wedge\bu_{k+1}$ in the standard basis \eqref{basis_wedge_k}. Equation~\eqref{cond2.15} now readily follows.
\end{proof}

\subsection{A result of Kleinbock and Margulis} \label{sec: KM}

We will deduce Theorem~\ref{new thm: BKM} from the following fundamental result of Kleinbock and Margulis \cite[Theorem 5.2]{KM}.

\begin{theoremKM}
Let $d,k\in \N, C,\alpha>0$, $0< \rho\leq 1/k$, and a ball $B_0\subset \R^d$ be given. Let $h:3^kB_0\to \mathrm{GL}_{k}(\R)$ be a continuous map. Suppose that for every integer $1\le r\le k$, for any primitive collection $\bv_1,\dots,\bv_r\in \Z^k$ of linearly independent integer vectors the following two conditions are satisfied:
    \begin{enumerate}
        \item $\bx\mapsto \Vert h(\bx)\bv_1\wedge \dots\wedge h(\bx) \bv_r\Vert$ is $(C,\alpha)$-good on $3^k B_0;$
        \item \label{thm KM}$\sup_{\bx\in B_0}\Vert h(\bx)\bv_1\wedge \dots\wedge h(\bx) \bv_r\Vert\geq \rho.$
    \end{enumerate}
Then for every $\varepsilon\leq \rho$,
$$\mu_{d}\{\bx\in B_0~|~\lambda_1(h(\bx)\Z^k)\leq \varepsilon\}\ll_{k,d} C \left(\frac{\varepsilon}{\rho}\right)^{\alpha}\mu_{d}(B_0),$$
where the implied constant depends on $k$ and $d$ only.
\end{theoremKM}

The definition of $(C,\alpha)$-good functions can be found in \cite{KM} and numerous subsequent papers, such as \cite{Kleinbock-exponent}. We recall the following result for later use.

\begin{lemma}[{\cite[Corollary 3.3]{Kleinbock-extremal}}]\label{lem: good func}
    Let $\bU\subset \R^d$ be a connected open set, $g=(g_i)_{i=1}^k:\bU\to\R^k$. Suppose that each $g_i$ is analytic. Then, for every $\bx_0\in \bU$ there exists a ball $B(\bx_0)\subset \bU$ and constants $C>0$ and $\alpha>0$ such that any linear combination of $1, g_1,\dots, g_k$ is $(C,\alpha)$ good at $B(\bx_0)$.
\end{lemma}

Lemma~\ref{lem: good func} will essentially verify condition~(1) of Theorem~KM in our applications. However, verifying condition~(2) requires considerably more work. To accomplish this task, we provide three auxiliary lemmas. We begin with the following well-known result, which can be found, for example, in \cite[Equation~(7.5)]{BKM}. It easily follows from the linear independence of the coordinate functions of $\tilde F$, defined in \eqref{Ftilde}.

\begin{lemma}\label{lem: just for rank 1} Let $F:\bU\to\R^n$ be nondegenerate at $\bx_0\in\bU$, and $\tilde F$ be as in \eqref{Ftilde}. Then there exists a neighborhood $V_1\ni\bx_0$ such that for any ball $B\subset V_1$, there exists $\rho_1>0$ such that for any $\ba\in\R^{n+1}$ with $\Vert \ba\Vert\ge1$, \begin{equation}\label{eq: rk 1} \sup_{\bx\in B}| \tilde F(\bx)\cdot\ba|\geq \rho_1. \end{equation} 
\end{lemma}

Verifying condition (2) of Theorem~KM will require us to deal with skew gradients \eqref{SG}. The following relevant result can be found in \cite[Proposition 4.1]{BKM}.

\begin{lemma}\label{LemSG}
Let $\bU\subset\R^d$, and $\bx_0\in \bU$, and $\mathcal{G}\subset C^{l}(\bU)$ be a compact family of maps $g:\bU\to \R^2 $ such that 
    \begin{equation}\label{eq: BKM skew requirement}
        \inf_{g\in \mathcal{G},\; \Vert \bv\Vert_2=1}\sup_{\vert \beta\vert\leq l}\vert \bv\cdot\partial_{\beta}g(\bx_0)\vert>0,
    \end{equation}
    then there exists $V\subset \bU$ containing $\bx_0$ such that 
    for every ball $B\subset V$ of $\bx_0$, there exists $\rho>0$, depending on $\mathcal{G}, B$ such that 
    $$
\sup_{\bx\in B}\Vert \tilde \nabla g(\bx)\Vert\geq \rho\quad \text{ for every } g\in \mathcal{G}.
    $$
\end{lemma}

The next result specialises Lemma~\ref{LemSG} to pairs of functions
$\tilde F(\bx)\cdot\ba$.

\begin{lemma}\label{lem:crucial nabla}
Let $F$, $\tilde F$, and $\bx_0$ be as in Lemma~\ref{lem: just for rank 1}. Then there exists a neighborhood $V_2\ni\bx_0$ such that for any ball $B\subset V_2$, there exists $\rho_2>0$ such that for any $\ba,\mathbf{b}\in\R^{n+1}$ with $\Vert \ba\wedge \mathbf{b}\Vert\ge1$ we have that
    \begin{equation}
        \sup_{\bx\in B}\Vert \tilde\nabla (\tilde F(\bx)\cdot\ba, \tilde F(\bx)\cdot\mathbf{b})\Vert\geq \rho_2.
    \end{equation}
    
\end{lemma}
\begin{proof}
By the nondegeneracy of $F$ at $\bx_0$, there exists $l\in\N$ such that $F\in C^l$ in a neighborhood of $\bx_0$ and \eqref{nondegen} holds. Consider the collection 
$$
\mathcal{G}:=\big\{(g_1,g_2)~\vert~g_i(\bx)=\tilde F(\bx)\cdot\bw_i\text{ for some } \bw_1,\bw_2\in \R^{n+1},\; 
   \Vert \bw_1\wedge\bw_2\Vert=1\big\}.
$$
By the compactness of the subset $\Vert \bw_1\wedge\bw_2\Vert=1$ on the Grassmanian, $\mathcal{G}$ is a compact family of $C^l$ maps defined on some neighborhood of $\bx_0$. Thus, the desired conclusion will follow from Lemma~\ref{LemSG}  once we verify \eqref{eq: BKM skew requirement}. 

For $i=1,2$ write $\bw_i=(w_{i,0},\bw_i')$ with $\bw_i'\in\R^n$.
Then, 
$$
g_i(\bx)=\tilde F(\bx)\cdot\bw_i=w_{i,0}+F(\bx)\cdot\bw_i'.
$$
We also have that 
$$
\partial_{\beta}(g_1, g_2)(\bx) =(\partial_{\beta}\tilde F(\bx)\cdot\bw_1, \partial_{\beta}\tilde F(\bx)\cdot\bw_2).
$$
Hence, for any $\bv=(v_1,v_2)\in \R^2$ with $\Vert \bv\Vert_2=1$, we have that
$$
\partial_{\beta}(g_1, g_2)(\bx_0)\cdot \bv= \partial_{\beta}\tilde F(\bx_0)\cdot (v_1\bw_1+v_2\bw_2).
$$   
Suppose that 
$$
\inf_{\bv\in\R^2,\; \Vert \bv\Vert_2=1,\; g\in \mathcal{G}} \;\;\sup_{\vert\beta\vert\le l}\vert \bv\cdot \partial_{\beta} g(\bx_0)\vert =0.
$$ 
Then, by compactness, there exists $\bv\in \R^2$ with $\Vert \bv\Vert_2=1$ and $\bw_1,\bw_2\in \R^{n+1}$ with $\Vert \bw_1\wedge\bw_2\Vert=1$ such that $$\sup_{\vert\beta\vert\le l}\vert  \partial_{\beta}\tilde F(\bx_0)\cdot (v_1\bw_1+v_2\bw_2)\vert =0.
$$
If $v_1\bw_1'+v_2\bw_2'\neq \mathbf{0}$, then this contradicts the nondegeneracy of $F$ at $\bx_0$. Therefore, $v_1\bw_1'+v_2\bw_2'=\mathbf{0}$, and the above equation implies that 
$v_1 w_{10}+v_2 w_{20}=0$. This implies that 
$v_1\bw_1=-v_2\bw_2$, and hence $\bw_1\wedge \bw_2=\vv0$, contrary to the assumption $\Vert \bw_1\wedge \bw_2\Vert=1.$  
\end{proof}

\section{Proof of Theorem~\ref{new thm: BKM}}

We begin with the following key proposition.

\begin{proposition}\label{prop: crucial}
    Let $F:\bU\to\R^n$ be as in \eqref{eq: F}, analytic and nondegenerate at $\bx_0\in\bU$.
Then there exists a neighbourhood $V$ of $\bx_0$ and $\alpha>0$ depending on $F$ and $\bx_0$ such that for any ball $B\subset V$ there exists $E>0$ such that for any diagonal matrix $b=\diag(b_0, b_1,\dots, b_n)\in \mathrm{GL}_{n+1}(\R)$ satisfying
\begin{equation}\label{eq: condition on b}
b_0\geq 1,\quad b_{d+1}=\dots=b_n\leq 1,\quad b_{1}=\dots=b_d\geq 1,
\end{equation}
and
\begin{equation}\label{b_0 is too large 2}
    b_0 b_1 b_{d+1}^m\geq 1,
\end{equation}
for any $\ep>0$, 
\begin{equation}
    \mu_d\{\bx\in B~|~\lambda_1(bu_1^\star(\bx)\Z^{n+1})<\ep\}\leq E \ep^\alpha \mu_d(B),
\end{equation}
where $u_1^\star$ is given by \eqref{eqn:dual u_1(x)}.

\end{proposition}

\begin{proof}
With the view to applying Theorem~KM stated in \S~\ref{sec: KM}, define 
$$
h(\bx):=b u_1^\star(\bx).$$ 
In what follows a ball $B(\bx_0)$ will be required to satisfy condition $3^{n+1}B(\bx_0)\subset V_1\cap V_2$, where $V_1$ and $V_2$ are as in Lemmas~\ref{lem: just for rank 1} and~\ref{lem:crucial nabla}. By Lemma~\ref{lem: good func}, there exists a a ball $B(\bx_0)$, and constants $C$ and $\alpha>0$ such that for any primitive collection $\bv_1,\dots,\bv_r\in \Z^{n+1}$, the map $\Vert h(\bx)\bv_1\wedge\dots\wedge h(\bx)\bv_r\Vert$ is $(C, \alpha)$ good on  $3^{n+1}B$, for any ball $B\subset B_0$.

Now fix a ball $B\subset B(\bx_0)$. To apply Theorem~KM, we need to check that there exists $\rho$ satisfying $0<\rho<1/(n+1)$, such that for any primitive collection $\bv_1,\dots,\bv_r\in\Z^{n+1}$,
\begin{equation}\label{want: lower bound on covol}
\sup_{\bx\in B}\Vert h(\bx)\bv_1\wedge\dots\wedge h(\bx)\bv_r\Vert\geq \rho.
\end{equation}
For $r=1$, this follows  from Lemma~\ref{lem: just for rank 1} and the fact that $b_0\geq 1$, for any $\rho\leq \rho_1$. 

Thus, from now on we let $r\geq 2$. As before, we denote the standard basis vectors in $\R^{n+1}$ by $\be_0,\be_1,\dots,\be_n$. Let  $\bv_1,\dots,\bv_r\in\Z^{n+1}$ be any primitive collection. Then, by Lemma~\ref{lem: dual prim}, there exists a primitive collection $\bu_1,\dots,\bu_{n+1-r}\in\Z^{n+1}$ satisfying \eqref{dot} and \eqref{same norm}.
Take any multi-index of the form
$$
I=\{0,i_1,\dots, i_{s+1}, j_1, \dots, j_{r-2-s}\},
$$ 
where 
\begin{equation}\label{cond_I}
0\le s\le r-2,\; 
\{j_1,\dots, j_{r-2-s}\}\subset \{d+1,\dots, n\},\; \{i_1,\dots,i_{s+1}\}\subset \{1,\dots, d\}.
\end{equation}
Then, by Lemma~\ref{lem: dual basis consq} and the Laplace  identity (see \cite[p.105]{MR1176315} or \cite[Equation~(3.3)]{Ber12}), the absolute value of the $I$-coordinate of the multivector appearing in \eqref{want: lower bound on covol}, with respect to the standard basis \eqref{basis_wedge_k}, is given by 
\begin{equation}\label{eq: covolume}
\begin{aligned}
    &\vert (h(\bx)\bv_1\wedge\dots\wedge h(\bx)\bv_r)_I\vert\\
    &=b_0 \prod_{k=1}^{s+1} b_{i_k}\prod_{k=1}^{r-2-s}b_{j_k}\vert \tilde F(\bx)\wedge_{k=1}^{s+1} \partial_{i_k}\tilde F(\bx)\wedge_{k=1}^{r-2-s} \be_{j_k}\wedge\bu_1\wedge \dots\wedge \bu_{n+1-r}\vert. 
\end{aligned}
\end{equation}

Let $s$ be the smallest integer in the range $0\le s\le r-2$ such that there exists a $\{j_1,\dots,j_{r-2-s}\}\subset \{d+1,\dots, n\}$ satisfying
\begin{equation}\label{eq: wedge nonzero}
\wedge_{k=1}^{r-2-s} \be_{j_k}\wedge\bu_1\wedge \dots\wedge \bu_{n+1-r}\neq \mathbf{0}.\end{equation} 
Note that such $s$ exists, since if $s=r-2$ then $\{j_1,\dots,j_{r-2-s}\}=\varnothing$ meets \eqref{eq: wedge nonzero} by the linear independence of the vectors $\bu_1,\dots,\bu_{n+1-r}$. We now consider two different scenarios for the value of $s$.

\medskip

\noindent\textbf{Case $1$:}  $s=0$. In this case, by Lemma~\ref{lem: dual prim}, there exists $\ba,\mathbf{b}\in\Z^{n+1}$ orthogonal to every vector appearing in \eqref{eq: wedge nonzero} such that 
$$
\Vert \ba\wedge \mathbf{b}\Vert_2= \Vert \wedge_{k=1}^{r-2} \be_{j_k}\wedge\bu_1\wedge \dots\wedge \bu_{n+1-r}\Vert_2.
$$ Now using Lemma~\ref{lem: dual basis consq}, we for any $i_1\in \{1,\dots, d\}$
$$
\begin{aligned}
&\vert \tilde F(\bx)\wedge \partial_{i_1}\tilde F(\bx)\wedge_{k=1}^{r-2} \be_{j_k}\wedge\bu_1\wedge \dots\wedge \bu_{n+1-r}\vert 
 =\vert (\tilde F(\bx)\wedge \partial_{i_1}\tilde F(\bx))\cdot (\ba\wedge \mathbf{b})\vert.
\end{aligned}
$$
By Lemma~\ref{lem:crucial nabla}, we know there exists $\rho_2>0$ such that 
$$\sup_{\bx\in B}\Vert \tilde\nabla (\tilde F(\bx)\cdot\ba, \tilde F(\bx)\cdot\mathbf{b})\Vert\geq \rho_2. $$
Here $\rho_2$ only depends on the ball $B$ and $F$, but not $\ba$ and $\vv b$. Hence, there exists $\bx_1\in B$, depending on $\ba$ and $\vv b$, such that 
$$\Vert\tilde\nabla (\tilde F(\bx_1)\cdot\ba, \tilde F(\bx_1)\cdot\mathbf{b})\Vert\geq \rho_2/2.$$ 

Note that for any $\bx\in B,$
$$
\Vert \tilde\nabla (\tilde F(\bx)\cdot\ba, \tilde F(\bx)\cdot\mathbf{b})\Vert=\max_{i=1}^d \Vert (\tilde F(\bx)\wedge \partial_{i}\tilde F(\bx))\cdot (\ba\wedge \mathbf{b})\Vert.
$$

Thus there exists there exists $1\leq i_1\leq d$ (depending on $\bx_1$ and $\ba,\mathbf{b}$) such that $$\Vert (\tilde F(\bx_1)\wedge \partial_{i_1}\tilde F(\bx_1))\cdot (\ba\wedge \mathbf{b})\Vert>\rho_2/2.$$
Then, by including this $i_1$ into the multi-index $I=\{0, i_1, j_1,\dots, j_{r-2}\}$ we obtain that
\begin{equation}\label{eq3.8}
    \sup_{\bx\in B}\Vert h(\bx)(\bv_1\wedge\dots\wedge\bv_r)\Vert\geq \vert h(\bx_1)(\bv_1\wedge\dots\wedge\bv_r)\vert_{I}\stackrel{\eqref{eq: condition on b}}{>} b_0 b_1 b_{d+1}^{m}\rho_2/2\stackrel{\eqref{b_0 is too large 2}}{>}\rho_2/2>0.
\end{equation}
Note that the $s=0$ case can only happen when $r\leq m+2$. Hence, in view of the assumption $b_{d+1}\leq 1$, imposed by \eqref{eq: condition on b}, we have the inequality $b_0 b_1 b_{d+1}^{r-2}\geq b_0 b_1 b_{d+1}^m$, which was used in \eqref{eq3.8} alongside \eqref{b_0 is too large 2}.

\medskip

\noindent\textbf{Case $2$:}
$s> 0$. By definition, there exist $j_1,\dots, j_{r-2-s}\in \{d+1,\dots, n\}$ such that \eqref{eq: wedge nonzero} holds. This immediately implies that $r-2-s\leq m=n-d$. By the minimality of $s$ and the assumption $s>0$, we have that for any $j\in J':=\{d+1,\dots,n\}\setminus\{j_1,\dots, j_{r-2-s}\}$ 
\begin{equation}\label{condition3.8}
\be_j\wedge \bigwedge_{k=1}^{r-2-s}\be_{j_k}\wedge\bu_1\wedge\dots\wedge\bu_{n+1-r}=\mathbf{0}.
\end{equation}

Observe that  $\{\be_j:j\in J'\}\cup\{\be_{j_1},\dots,\be_{r-2-s}\}=\{\be_{d+1},\dots,\be_n\}$ is a primitive collection. Also, note by \eqref{cond_I} that $d-1-s\ge0$.
Then, in view of \eqref{eq: wedge nonzero} and \eqref{condition3.8}, by Lemma~\ref{lem: helpful-general}, there exist $\bu_1',\dots,\bu'_{d-1-s}\in \Z^{n+1}$ such that
\begin{equation}\label{def: bu'}
\bigwedge_{k=1}^{r-2-s} \be_{j_k}\wedge\bu_1\wedge \dots\wedge \bu_{n+1-r}=\ell\bigwedge_{j=d+1}^n \be_{j}\wedge\bu_1'\wedge \dots\wedge \bu'_{d-1-s},
\end{equation}
where $\ell$ is a non-zero integer. 

By the definition of $s$, \eqref{eq: wedge nonzero} and \eqref{def: bu'}, there exist $i_1,\dots, i_{s+1}\in \{1,\dots, d\}$ such that 
\begin{equation}\label{eq: nonzero wedge 2}
\wedge_{k=1}^{s+1} \be_{i_k}\wedge\be_{d+1}\wedge\dots\wedge\be_{n+1}\wedge\bu_1'\wedge\dots\wedge\bu'_{d-1-s}\neq \mathbf{0}.
\end{equation} 
Let $I=\{0, i_1,\dots, i_{s+1}, j_1,\dots, j_{r-2-s}\}$. Note here $I$ only depends on $\bv_1,\dots,\bv_r$. Then, using \eqref{def: bu'} we get that $\vert (h(\bx)\bv_1\wedge\dots\wedge h(\bx)\bv_r)_I\vert$ is equal to
$$\begin{aligned}
     &\vert \tilde F(\bx)\wedge_{k=1}^{s+1} \partial_{i_k}\tilde F(\bx)\wedge_{k=1}^{r-2-s} \be_{j_k}\wedge\bu_1\wedge \dots\wedge \bu_{n+1-r}\vert\\
     &= \vert \ell \tilde F(\bx)\wedge_{k=1}^{s+1} \partial_{i_k}\tilde F(\bx)\wedge_{j=d+1}^{n+1} \be_{j}\wedge\bu'_1\wedge \dots\wedge \bu'_{d+1-s}\vert.
\end{aligned}$$
Using the fact that $F(\bx)=(\bx, f(\bx))$, 
the above becomes
\begin{equation}\label{eq: 1 row imp}
\vert \ell\tilde F(\bx)\wedge_{k=1}^{s+1} \be_{i_k}\wedge_{k=d+1}^{n+1} \be_{j}\wedge\bu'_1\wedge \dots\wedge \bu'_{d+1-s}\vert.
\end{equation}

By \eqref{eq: nonzero wedge 2}, Lemma~\ref{lem: dual prim} guarantees that there exists $\ba\in\Z^{n+1}\setminus \{\mathbf{0}\}$, which is orthogonal to all the vectors appearing in \eqref{eq: nonzero wedge 2}, and 
$$
\Vert \ell\wedge_{k=1}^{s+1} \be_{i_k}\wedge\be_{d+1}\wedge\dots\wedge\be_{n+1}\wedge\bu_1'\wedge\dots\wedge\bu'_{d-1-s}\Vert_2=\Vert \ba\Vert_2.$$
Thus using Lemma~\ref{lem: dual basis consq} we can rewrite  \eqref{eq: 1 row imp} as follows
$$
\vert \ell\tilde F(\bx)\wedge_{k=1}^{s+1} \be_{i_k}\wedge_{k=d+1}^{n+1} \be_{j}\wedge\bu'_1\wedge \dots\wedge \bu'_{d+1-s}\vert=\vert \tilde F(\bx)\cdot\ba\vert.$$
Now using \eqref{eq: covolume} and Lemma~\ref{lem:crucial nabla}, we get that
 \begin{equation}
     \sup_{\bx\in B}\Vert h(\bx)(\bv_1\wedge\dots\wedge\bv_r)\Vert>b_0 b_1^{s+1} b_{d+1}^{r-s-2}\rho_1> b_0 b_1 b_{d+1}^{m}\rho_1\stackrel{\eqref{b_0 is too large 2}}{>}\rho_1>0.
 \end{equation}
Here, we used the fact that $r-s-2\leq m$, and that $b_{d+1}\leq 1$.

Finally taking $\rho=\min\{\rho_1,\rho_2/2, 1/(n+1)\}$ verifies \eqref{want: lower bound on covol}.
\end{proof}

\begin{remark}
    In the above proposition, $b_{d+1}\leq 1$ is imposed just to make this an interesting case. In fact, when $b_{d+1}\geq 1$, then all eigenvalues are $\geq 1$, and the conclusion of the proposition is also true. 
\end{remark}

\subsection{Proof of Theorem~\ref{new thm: BKM}}
We will apply Proposition~\ref{prop: crucial}, with 
$$
b_0=\delta^{-1}(\delta KT^m)^{1/(m+2)}, \qquad b_1=\dots=b_d= K^{-1}(\delta K T^m)^{1/(m+2)}
$$ 
and 
$$
b_{d+1}=\dots=b_{n+1}= T^{-1}(\delta K T^m)^{1/(m+2)}.$$ 
Conditions \eqref{eq: condition on b} imposed on $b$ within Proposition~\ref{prop: crucial} are readily verified using \eqref{condition on K T}. Also note that 
$$
b_0 b_1 b_{d+1}^{m}= \delta^{-1} K^{-1} T^{-m} (\delta K T^{m})^{(m+2)/(m+2)}=1,
$$ 
thus \eqref{b_0 is too large 2} is satisfied. Next, using the form of $F$ given by \eqref{eq: F} and the definitions of $b$ and $u_1^*$ we observe that
$$
\Deltanew(\delta, K, T)\subset \{\bx\in \bU~|~\lambda_1(bu_1^\star(\bx)\Z^{n+1})\leq (\delta K T^m)^{1/{(m+2)}}\}.
$$ 
Now Theorem~\ref{new thm: BKM} follows from Proposition~\ref{prop: crucial}.

\section{Counting rational points}

Throughout this section $\lambda_1(\Lambda),\dots,\lambda_{n+1}(\Lambda)$ will denote successive minima of a lattice $\Lambda\subset\R^{n+1}$.
Let $\bU\subset \R^d$ and the map $F$ be of the same form as before, that is \eqref{eq: F}, and regarded as a parameterisation of a given manifold. For any subset $\Delta\subset \bU$, define
$$N^\star(\Delta;\ep, t):=\#\Big\{(q, \bp)\in \Z^{n+1}_{\rm prim}~|~\exists~ \bx\in \Delta \text{ such that }\begin{array}{l}
     \Vert qF(\bx)-\bp\Vert<\ep,  \\[0.5ex]
     e^t\leq \vert q\vert< e^{t+1} 
\end{array} \Big\}.$$
Thus, if $\cM=F(\bU)$, $N(\ep,t)=N^\star(\bU;\ep,t).$
We also define the following modified counting function in which the distance to $\cM$ is taken across a particular directon:
$$
N(\Delta;\ep, t):=\#\mathcal{N}(\Delta;\ep, t)
$$
where
\begin{equation}\label{NN}
\mathcal{N}(\Delta;\ep, t)=\begin{aligned}\left\{(q, \ba,\bb)\in \Z^{n+1}_{\rm prim}~|~\begin{array}{l}
\ba/q\in \Delta,\\[0.5ex] 
\Vert qf(\ba/q)-\bb\Vert\leq \ep,\\[0.5ex] 
e^t\leq \vert q\vert< e^{t+1}
\end{array}\right\},
\end{aligned}
\end{equation}
where $\ba\in\Z^d$ and $\bb\in\Z^m$.
It is readily seen that 
$$
N(\Delta;\ep, t)\leq N^\star(\Delta;\ep, t)
$$
and
$$
N(\Delta_1\cup\Delta_2;\ep, t)=N(\Delta_1;\ep, t)+N(\Delta_2;\ep, t)
$$
whenever $\Delta_1,\Delta_2\subset \bU$ are disjoint. 

Let 
\begin{equation}\label{g}
g:=\diag(\underbrace{\ep^{-1} }_{\text{$m$ copies}},\underbrace{\ep^{-1/2}e^{-t/2}}_{\text{$d$ copies}}, e^{-t})\in \mathrm{GL}_{n+1}(\R),
\end{equation}
and define the set 
$$
\mathfrak{M}_0(\varepsilon,t):=\Big\{\bx\in \bU~|~\lambda_{n+1}(gu_1(\bx)\Z^{n+1})\ge\frac1{(n+1)^{3/2}}\Big\},
$$ where $u_1(\bx)$ is defined in Section~\ref{sec: notation}.
Then by Mahler's duality see \cite[Theorem~3.2 and Lemma~3.3]{BY}, we have that
\begin{equation}\label{eq: dual definition of special}
\mathfrak{M}_0(\varepsilon,t)\subseteq\{\bx\in \bU~|~\lambda_{1}(g^\star u_1^\star(\bx)\Z^{n+1})\ll1\},
\end{equation}
where the implied constant depends on $n$ only. Define the special part as the $\ep^{1/2}e^{-t/2}$-thickening of 
$\mathfrak{M}_0(\varepsilon,t)$:
\begin{equation}\label{def: minor}\mathfrak{M}(\varepsilon,t):= \bigcup_{\bx\in \mathfrak{M}_0(\varepsilon,t)} B(\bx, \ep^{1/2}e^{-t/2})\cap\bU.\end{equation} 
We proceed with the following analogue of \cite[Proposition 5.1]{BY}.

\begin{lemma}\label{lem: minor is small}
Let $F$ be an analytic map of the form \eqref{eq: F},  nondegenerate at $\bx_0\in\bU$. Let $\kappa>0$. Then there exists a ball $B_0$ centred at $\bx_0$ and constant $c=c(B_0, F, n, d,\kappa)>1$ such that for any $t\geq 1$ and any $\ep\in(0,1)$ satisfying \begin{equation}\label{lower bound on ep}
    \ep> c\,e^{-\frac{3t}{2m+1}},
\end{equation}
we have that
$$\mu_d(\mathfrak{M}(\varepsilon,t)\cap B_0)\leq  \kappa\mu_d(B_0).
$$ Moreover, $\mathfrak{M}(\varepsilon,t)$ can be written as a union of balls of radius $\ep^{1/2}e^{-t/2}$ of intersection multiplicity at most $N_d$, the Besicovitch constant.
\end{lemma}
\begin{proof}
    From \cite[Lemma 6.7]{BD}, we get that 
    $$\mathfrak{M}(\varepsilon,t)\subset \Deltanew(c_2\delta, c_2K, c_2T),$$ where $\delta=e^{-t}, K=(\ep e^t)^{-1/2}, T=\ep^{-1}$ and $c_2>0$ depends on $F$ and $n$. 
    Note that $$\delta K T^m= e^{-3t/2}\ep^{-(m+1/2)}.$$ 
Since trivially $\ep<1\le e^{t}$, we have that $\delta<K$. Further, since $m\ge1$, one can easily see that $\delta^{m+2}\le K^{m+2}\le\delta K T^{m}$. Also, by \eqref{lower bound on ep}, $K\le1$ and so $\delta K\le T^{2}$. Thus, conditions \eqref{condition on K T} hold, and Theorem~\ref{new thm: BKM} is applicable. Hence, there exists $E_{B_0}$ such that 
$$\mu_{d}(\mathfrak{M}(\ep,t)\cap B_0)\leq E_{B_0} (e^{-3t/2} \ep^{-1/2-m})^{\alpha/(m+2)}\mu_d(B_0)\le E_{B_0} \left(\frac 1c\right)^{\alpha\frac{2m+1}{2m+4}}\mu_d(B_0),$$ hence the conclusion follows on taking $c$ large enough.

The moreover part of the statement follows from the definition of $\mathfrak{M}(\ep,t)$ and the Besicovitch covering theory, see \cite{BY} for a similar full argument.
\end{proof}
\begin{remark}\label{rem: just nondeg}
In the proof of Lemma~\ref{lem: minor is small}, using the original \textit{BKM estimate} \cite[Theorem 1.4]{BKM} instead of Theorem~\ref{new thm: BKM} gives the bound 
$$\mu_{d}(\mathfrak{M}(\ep,t)\cap B_0)\leq E_{B_0}(e^{-3t/2} \ep^{-1/2-(n-1)})^{\alpha_1},$$ for some (explicitly computable) $\alpha_1>0.$ Hence the same conclusion as in Lemma~\ref{lem: minor is small} can be obtained for
\begin{equation}\label{eq: just for nondeg}
    \ep> c\,e^{-\frac{3t}{2n-1}},\end{equation} without analyticity or $C^\infty$ assumption on $F$. 
\end{remark}

\subsection{Partitioning $B_0$} Let $B_0\subset \bU$ be as in Lemma~\ref{lem: minor is small}. Take a maximal (in number of elements) collection $\{\Delta_i\}$ of open disjoint cubes of side length  $\ep^{1/2}e^{-t/2}$ such that $\bigcup_{i} \Delta_i\subset \frac{1}{2}B_0$. It is readily seen that the number of different cubes $\Delta_i$ is 
\begin{equation}\label{eq4.5}
\# \{\Delta_i\}\ge c_d\ep^{-d/2}e^{td/2}\mu_d(B_0).
\end{equation}
We will call a cube $\Delta_i$ \textit{good} if $\Delta_i\cap \bU \setminus \mathfrak{M}(\ep, t)\neq \emptyset$. 
Define the number of good cubes: 
$$
K(\ep,t):=\# \{\Delta_i~|~\Delta_i\cap \bU \setminus \mathfrak{M}(\ep, t)\neq \emptyset\}.$$

\begin{lemma}\label{lem: mostly good}
Under the assumptions of Lemma~\ref{lem: minor is small}, for any $0<\kappa<c_d$ with $c_d$ as in \eqref{eq4.5}, for all $t$ large enough, and $0<\ep<1$ satisfying \eqref{lower bound on ep}, we have that
$$
    K(\ep,t)\ge(c_d-\kappa)\ep^{-d/2}e^{td/2} \mu_d(B_0).
$$
\end{lemma}
\begin{proof}
First, observe that if $\Delta_i$ is not \textit{good} then $\Delta_i\subset\mathfrak{M}(\ep, t)$. Since the cubes $\Delta_i$ are disjoint, the number of cubes that are not good is bounded by the measure of $\mathfrak{M}(\ep, t)$ divided by the measure of a single cube. Then, by Lemma~\ref{lem: minor is small} this number 
$$
\le \kappa \ep^{-d/2}e^{td/2}\mu_d(B_0) .
$$
Hence, using \eqref{eq4.5} we get that
$$
K(\ep,t)\ge c_d\ep^{-d/2}e^{td/2}\mu_d(B_0) - 
\kappa \ep^{-d/2}e^{td/2}\mu_d(B_0) =  (c_{d}- \kappa)\ep^{-d/2}e^{td/2}\mu_d(B_0)
$$
as required.
\end{proof}

The following lemma can be viewed as a modification of \cite[Lemma~2.1]{BVVZ21}.

\begin{lemma}\label{lem: one good point}
   Suppose $B_0\subset \bU$ be as in Lemma~\ref{lem: minor is small}, then for any  
    cube $\Delta\subset B_0$ of side length $\ep^{1/2} e^{-t/2}$,
    and for any $c_2\ge2$,
$\bU\setminus \mathfrak{M}_0(\ep,t)\cap \Delta$ is a subset of    
\begin{equation}\label{eq4.8++}
\left\{\bx\in \Delta~\Big|~\exists (q, \ba, \bb)\in \Z^{n+1}\; \begin{aligned}
&\textstyle \vert q f_{j}(\bx)-b_j-\sum\nolimits_{i=1}^d(qx_i-a_i)\partial_{i}f_{j}(\bx)\vert\leq \ep,\\
& \Vert q\bx-\ba\Vert \leq \ep^{1/2} e^{t/2},\\
& (c_2-1)e^t \leq  q \leq  (c_2+1) e^t
\end{aligned}\right\}.
\end{equation}
\end{lemma}

\begin{proof}
Let $\bx\in (\bU\setminus \mathfrak{M}_0(\ep, t)) \cap \Delta$.  Then $$\lambda_{n+1}(g u_1(\bx)\Z^{n+1})\le \frac1{(n+1)^{3/2}}.
$$
This implies there exists linearly independent vectors $\bv_1,\dots,\bv_{n+1}\in g u_1(\bx)\Z^{n+1} $ lying in the ball centred at $\vv0$ of radius $(n+1)^{-3/2}$. Thus, 
\begin{equation}\label{eq4.6+}
\|\bv_i\|\le \sqrt{n+1}\|\bv_i\|_2\le(n+1)^{-1}
\end{equation}
for each $i$. Define the vector
$$ \bw:=w_0(
          -f_{m}(\bx),
           \dots,
           -f_{1}(\bx),
           -x_{d},
           \dots,
           -x_1,
           1)^T \in\R^{n+1}, \quad \text{where } w_0= c_2e^{t}.
$$
Let $\eta_1,\dots,\eta_{n+1}$ be the real parameters obtained by solving the system 
$$g u_1(\bx)\bw=\sum_{i=1}^{n+1} \eta_i \bv_i.
$$
Their existence follows from the linear independence of $\bv_1,\dots,\bv_{n+1}$. Let $t_i\in \Z$ be such that $\vert \eta_i-t_i\vert \leq 1$ and let $\bu=(-b_m, \dots,-b_{1},-a_d,\dots, -a_{1}, q)^T\in \Z^{n+1}$ be defined from 
$$ 
g u_1(\bx)\bu=\sum_{i=1}^{n+1} t_i \bv_i.$$
Let $\ba=(a_1,\dots,a_d)$ and $\bb=(b_1,\dots,b_m)$. From the definition of $\eta_i$ and \eqref{eq4.6+} we get that 
\begin{equation}\label{eq4.7}
\Vert gu_1(\bx)(\bu-\bw)\Vert=\Vert \sum_{i=1}^{n+1} (t_i-\eta_i) \bv_i\Vert \leq 1 .
\end{equation}
Using the last coordinates of the vector on the left of \eqref{eq4.7} we get that
$$\begin{aligned}
& \vert q-w_0\vert\leq e^{t}.\end{aligned}
$$
Hence, since $w_0=c_2w^t$ and $c_2\ge2$ we get that 
\begin{equation}\label{eq4.8}
e^t\le (c_2-1)e^t\leq q\leq (c_2+1)e^t.
\end{equation}

Using the $d$ coordinates next to the last last of \eqref{eq4.7} we get that
$$
\ep^{-1/2} e^{-t/2}\vert (-a_i+w_0 x_i)+ (q-w_0)x_i\vert\leq 1,\quad(1\le i\le  d)
$$
and so
\begin{equation}\label{eq4.9}
\Vert q\bx-\ba\Vert \leq \ep^{1/2} e^{t/2}.
\end{equation}
Now from the first $m$ coordinates of \eqref{eq4.7}, we get that 
$$
\vert (-b_j+w_0 f_{j}(\bx))-\sum_{i=1}^d (-a_i+w_0x_i)\partial_{i}f_{j}(\bx)+(q-w_0)(f_{j}(\bx)-\sum_i \partial_{i}f_{j}(\bx)x_i)\vert\leq \ep
$$
and therefore for each $1\le j\le m$ we have that
$$
\vert q f_{j}(\bx)-b_j-\sum_{i=1}^d(qx_i-a_i)\partial_{i}f_{j}(\bx)\vert\leq \ep.
$$
This completes the proof.
\end{proof}

\begin{remark}
Note that the above lemma guarantees that for any analytic non-degenerate manifold there exists a rational point of denominator $\ll e^t$ lying $\ll \epsilon e^{-t}$ close to the manifold. However, it does not have the primitivity of $(q,\ba,\bb)$ and hence we cannot derive the desired lower bound, nor we can even ensure that the aforementioned rational point has denominator $\asymp e^t$. In the next proposition we resolve these issues. The proof  will require the following axillary statement from analytic number theory, which is well known is one form on another, for example see \cite[Theorem~3.9]{MR434929}. 
\end{remark}

\begin{lemma}\label{gcd}
Let $0\le x_0<x$, $0\le y_0 \le y$, $x\ge2$ and $y\ge2$. Let $N(x,y;x_0,y_0)$ be the number of coprime integer pairs $(a,b)$ such that $x_0 < a \le x$, $y_0 < b \le y$.
Then,
$$
N(x,y;x_0,y_0)=\frac{6}{\pi^2}(x-x_0)(y-y_0)+O\big((x+y)\min\{\log x,\log y\}\big),
$$
where the implied constant is absolute.
\end{lemma}

\begin{proof}
This is a relatively straightforward generalisation of \cite[Theorem~3.9]{MR434929}. For completeness we provide a proof. Without loss of generality we can assume that $y\le x$. For any integer $q\ge 1$ and $x\ge1$, let
$$
\varphi(x;q):=\#\{1\le k\le x\;\;\vert\;\gcd(k,q)=1\}.
$$
Then, by \cite[Exercise~9, Chapter~2]{MR434929},
$$
\varphi(x;q)=\sum_{d\mid q}\mu(d)\left[\frac{x}{d}\right].
$$
Consequently,
$$
N(x,y;x_0,y_0)=\sum_{y_0<q\le y}\Big(\varphi(x;q)-\varphi(x_0;q)\Big)
=\sum_{y_0<q\le y}\sum_{d\mid q}\mu(d)\Big(\left[\frac{x}{d}\right]-\left[\frac{x_0}{d}\right]\Big)
$$
$$
=(x-x_0)\sum_{y_0<q\le y}\sum_{d\mid q}\frac{\mu(d)}{d}+O\left(\sum_{1\le q\le y}d(q)\right)=
$$
by \cite[Theorem~2.3 and Theorem~3.3]{MR434929}
$$
=(x-x_0)\sum_{y_0<q\le y}\frac{\varphi(q)}{q}+O\left(y\log y\right)=
$$
by \cite[Exercise~5, Chapter~3]{MR434929}
$$
=(x-x_0)\left(\frac{6}{\pi^2}(y-y_0)+O(\log y)\right)+O\left(y\log y\right)
$$
$$
=\frac{6}{\pi^2}(x-x_0)(y-y_0)+O\left((x+y)\log y\right).
$$
Since $y\le x$, $\log y=\min\{\log x,\log y\}$ and the above completes the proof.
\end{proof}

\medskip

We will also use the following lemma on properties of a reduced basis of a lattice.

\begin{lemma}\label{reduced_basis}
Let $\Lambda$ be a lattice in $\R^{n+1}$ and $\lambda_1,\dots,\lambda_{n+1}$ denote the successive minima of $\Lambda$ with respect to the unit ball. Then there exists a basis $\bb_1,\dots,\bb_{n+1}$ of $\Lambda$ and a constants $k_0,\dots,k_{n+1}>0$ depending on $n$ only such that
\begin{equation}\label{lattice1}
\lambda_i\le\|\bb_i\|\le k_i\lambda_i\qquad(1\le i\le n+1)
\end{equation}
and for any $\bu\in\Lambda$ we have that
\begin{equation}\label{norm1}
k_0\sum_{i=1}^{n+1}|u_i|\|\bb_i\| \le \|\bu\|\le \sum_{i=1}^{n+1}|u_i|\|\bb_i\|,
\end{equation}
where $u_i$ are the coordinates of $\bu$ in this basis, that is the unique integers such that
\begin{equation}\label{basis2}
\bu=\sum_{i=1}^{n+1}u_i\bb_i.
\end{equation}
\end{lemma}

The construction of a basis as above falls into the reduction theory and is well known. For completeness, we provide a proof of the lemma using the most basic approach which is not optimizing the constants $k_i$ in any way.

\begin{proof}
By the definition of successive minima, there is a collection of linearly independent vectors $\ba_1,\dots,\ba_{n+1}\in\Lambda$ such that $\|\ba_i\|_2=\lambda_i$, see \cite[Lemma~1 on p.\;204]{MR1434478}. The existence of a basis satisfying \eqref{lattice1} is then an immediate consequence of \cite[Theorem~1(B) and Corollary~1 of \S{}I.2.2]{MR1434478}.

Regarding \eqref{norm1}, the right hand side is simply the triangle inequality applied to \eqref{basis2}. The converse follows from Minkowski's second theorem on successive minima, which together with \eqref{lattice1} implies that $\|\bb_1\|\cdots\|\bb_{n+1}\|\asymp\det(\Lambda)$ with the implied constants depending on $n$. Indeed, \eqref{basis2} can be rewritten as
\begin{equation}\label{new_eq}
\bu=\sum_{i=1}^{n+1}\tilde u_i\tilde \bb_i
\end{equation}
with $\tilde u_i=\lambda_iu_i$ and $\tilde \bb_i=\lambda_i^{-1}\bb_i$. Letting $\tilde\bu:=(\tilde u_1,\dots,\tilde u_{n+1})^T$ and $\tilde B:=[\tilde\bb_1,\dots,\tilde\bb_{n+1}]$, where $\tilde\bb_i$ are the columns of $\tilde B$, we get from \eqref{new_eq} that
\begin{equation}\label{neq_eq2}
\tilde\bu=\tilde B^{-1}\bu.
\end{equation}
By \eqref{lattice1}, every entry of $\tilde B$ is $\ll1$ in absolute value. Since we also have that 
$$|\det\tilde B|=\frac{|\det[\bb_1,\dots,\bb_{n+1}]|}{\lambda_1\dots\lambda_{n+1}}=\frac{\det(\Lambda)}{\lambda_1\dots\lambda_{n+1}}\asymp1,
$$
the entries of $\tilde B^{-1}$ are also $\ll1$. Hence, by \eqref{neq_eq2}, 
$$
\sum_i|u_i|\|\bb_i\|\asymp\max\{|u_i|\lambda_i\}\asymp\|\tilde\bu\|\ll \|\bu\| 
$$
whence the left hand side of \eqref{norm1} readily follows.
\end{proof}

\medskip

\begin{proposition}\label{prop: count on major}
There is $t_0$ such that for $t>t_0$, any $\ep\in(0,1)$ satisfying \eqref{lower bound on ep}, and any good $\Delta_i$
$$
N(\Delta_i;\ep,t)\gg_{n,d, F}  \ep^m e^{t(d+1)} \left(\frac{\ep}{e^t}\right)^{d/2}.$$
\end{proposition}

\begin{proof}
Let $\bx_0$ be the centre of $\Delta_i$, 
$$
\Lambda=g u_1(\bx_0)\Z^{n+1},
$$
and $\lambda_1\le\dots\le\lambda_{n+1}$ be the successive minima of $\Lambda$ with respect to the unit ball. Let $\bb_1,\dots,\bb_{n+1}$ be the (reduced) basis arising from Lemma~\ref{reduced_basis}. 

Since $\Delta_i$ is good, there exists $\bx\in\Delta_i\setminus \mathfrak{M}(\ep,t)$. In particular, $\Vert \bx-\bx_0\Vert\le \left(\frac{\ep}{e^t}\right)^{1/2}$. By the definition of $\mathfrak{M}(\ep,t)$, $\bx_0\notin \mathfrak{M}_0(\ep,t)$, which means that $\lambda_{n+1}<(n+1)^{-3/2}$. Hence,
\begin{equation}\label{eq: x_0 is good}
\lambda_1\le\dots\le\lambda_{n+1}<1.\end{equation} 
Let 
$$
\bs:=g u_1(\bx_0)(-b_m,\dots,-b_1,-a_d,\dots,-a_1,q)^T\in\Lambda,
$$ 
where $(q,\ba,\mathbf{b})$ arises from Lemma~\ref{lem: one good point}. Observe that $\|\bs\|\asymp c_2$, where the implied constants depend on $n$ and $F$. Expand $\bs$ in the above basis from Lemma~\ref{reduced_basis}:
\begin{equation}\label{eq:w}
\bs=\sum_{i=1}^{n+1} s_i \bb_i.
\end{equation}
By replacing $\bb_i$ with $-\bb_i$ if necessary we can assume without loss of generality that all $s_i$ are non-negative.
By \eqref{norm1} and the choice of $(q,\ba,\mathbf{b})$, we have that $\max_i|s_i|\lambda_i\asymp \|\bs\|\ll c_2$ and thus
\begin{equation}\label{sbound}
|s_i|\ll \lambda_i^{-1}c_2 \quad\text{for each $i$}.
\end{equation}
By \eqref{g} and Minkowski's second theorem, we have that 
\begin{equation}\label{det(g)}
\det(\Lambda)=\deg(gu_1(\bx_0))=\det(g)=\ep^{-m-d/2} e^{-(d/2+1)t}\asymp \lambda_1\cdots\lambda_{n+1}.
\end{equation}
Let $c_1:=\frac15c_2$ and assume that $c_1>1$. Define $B$ to be the cube $[-c_1,c_1]^{n+1}$.
Let $\bv\in\Lambda\cap B$. Then
$$
\bv=g u_1(\bx_0)(-b_m',\dots,-b_1,-a_d',\dots,-a_1',q')^T\in\Lambda\cap[-c_1,c_1]^{n+1}
$$
for some point $(-b_m',\dots,-b_1,-a_d',\dots,-a_1',q')\in\Z^{n+1}$. 
Using the explicit form of $g$ and $u_1(\bx_0)$ we get that
\begin{equation}\label{eq: q_1}
\begin{aligned}
    &\vert q' f_{j}(\bx_0)-b'_j-\sum_{i=1}^d(q'x_{0,i}-a'_i)\partial_{i}f_{j}(\bx_0)\vert\leq c_1\ep,\quad(1\le j\le m)\\
    & \Vert q'\bx_0-\ba'\Vert\leq c_1\ep^{1/2} e^{t/2},\\[1ex]
    & |q'|\leq c_1e^t,
\end{aligned}\end{equation}
where $\ba'=(a'_1,\dots,a'_d)$ and $\bb'=(b'_1,\dots,b'_m)$.
Furthermore, by the choice of $(q,\ba,\bb)$ made using Lemma~\ref{lem: one good point}, we also have that 
\begin{equation}\label{eq: q good}
\begin{aligned}
    &\vert q f_{j}(\bx_0)-b_j-\sum_{i=1}^d(qx_{0,i}-a_i)\partial_{i}f_{j}(\bx_0)\vert\leq \ep,\quad(1\le j\le m)\\
    & \Vert q\bx_0-\ba\Vert\leq \ep^{1/2} e^{t/2},\\[1ex]
    & (c_2-1)e^t\leq  q\leq (c_2+1)e^t.
\end{aligned}\end{equation}

Set $(q'',\ba'',\bb'')=(q',\ba',\bb')+ (q, \ba,\mathbf{b})$. Then combining \eqref{eq: q_1} and \eqref{eq: q good} gives
\begin{equation}\label{eq4.18}
\begin{aligned}
    &\vert q'' f_{j}(\bx_0)-b''_{j}-\sum_{i=1}^d(q''x_{0,i}-a''_{i})\partial_{i}f_{j}(\bx_0)\vert\leq 2c_1\ep,\\
    & \Vert q''\bx_0-\ba''\Vert\leq 2c_1\ep^{1/2} e^{t/2},\\[1ex]
    & (c_2-c_1-1)e^t\leq q''\leq (c_2+c_1+1)e^t.
\end{aligned}
\end{equation}
Here, $\ba''=(a_1'',\dots,a_d'')$ and $\bb''=(b_1'',\dots,b_m'')$.
Since $c_2=5c_1$ and $c_1>1$, we have that 
\begin{equation}\label{q}
3c_1e^t\le q''\le 7c_1e^t
\end{equation}
and therefore, by the middle inequality of \eqref{eq4.18},
$$
\|\bx_0-\ba''/q''\|\le (\ep e^{-t})^{1/2}.
$$
It follows that  $\ba''/q''\in\Delta_j$.

Using Taylor's expansion and \eqref{eq4.18} we get that 
\begin{equation}\label{ep}
\Vert q'' f(\ba''/q'')-\bb''\Vert\ll_{F, c_1} \ep.
\end{equation}
In view of \eqref{q} and \eqref{ep}, if $(q'',\ba'',\bb'')$ is primitive then
$$
(q'',\ba'',\bb'')\in \mathcal{N}(\Delta;\ep'',t'')\quad\text{with $e^{t''}=7c_1e^t$ and $\ep''\asymp_{n,F,c_1}\ep$},
$$
where $\mathcal{N}(\Delta;\ep'',t'')$ is defined as in \eqref{NN}.
Therefore, the statement of the proposition will follow (up to the  change of parameters $t\mapsto t''$, $\ep\mapsto\ep''$) if for a sufficiently large $c_1$ depending on $n$ and $F$ we can identify $\ep^m e^{(d+1)t}(\ep/e^t)^{1/2}$ primitive points $(q'',\ba'',\bb'')$ as above.

\bigskip

\noindent\textbf{Counting primitive points.} 
Recall that $(q'',\ba'',\bb'')=(q',\ba',\bb')+ (q, \ba,\mathbf{b})$. Since the primitivity of a point in a lattice is invariant under applying the linear transformation $gu_1(\bx_0)$, $(q'',\ba'',\bb'')$ is primitive if and only if 
$$
gu_1(\bx_0)(q',\ba',\bb')^T+ gu_1(\bx_0)(q, \ba,\mathbf{b})^T=\bv+\bs
$$ 
is primitive in $\Lambda$,
where 
$\bv\in B\cap\Lambda$,
$B=[-c_1,c_1]^{n+1}$, and $\bs$ is fixed at the beginning of the proof.
Write $\bv$ in the  basis $\bb_1,\dots,\bb_{n+1}$:
\begin{equation}\label{v}
\bv=\sum_{i=1}^{n+1} v_i \bb_i.
\end{equation}
Let 
\begin{equation}\label{eta_i}
1\le \eta_i \asymp c_1/\lambda_i,
\end{equation} 
where the implied constants depend on $n$. 
Then, by \eqref{norm1}, if we choose $(v_1,\dots,v_{n+1})\in\Z^{n+1}$ so that
\begin{equation}\label{range}
1\le v_i\le\eta_i,
\end{equation}
with parameters $\eta_i$ chosen by \eqref{eta_i} with suitably small implied constants, we can guarantee that $\bv\in B$. 
Further, by \eqref{eq:w} and \eqref{v},
\begin{equation}
\bv+\bs=\sum_{i=1}^{n+1} (v_i+s_i) \bb_i.
\end{equation}
Then, $\bv+\bs\in\Lambda$ is primitive if and only if
$\gcd(v_1+s_1,\dots,v_{n+1}+s_{n+1})=1$. For its primitivity it is sufficient that $\gcd(v_1+s_1,v_{2}+s_{2})=1$. Therefore, by \eqref{eta_i} and \eqref{range},
\begin{equation}\label{eq4.15}
\#\{\bv\in B\cap\Lambda:\bv+\bs\text{ is primitive}\}\ge N(x,y;x_0,y_0)\prod_{i=3}^{n+1}\eta_i
\end{equation}
where the implied constant depends on $n$ and $N(x,y;x_0,y_0)$ is as in Lemma~\ref{gcd} with
\begin{align*}
&x=s_1+\eta_1, && y=s_2+\eta_2;\\ 
&x_0=s_1,&& y_0=s_2.&
\end{align*}
By Lemma~\ref{gcd},
$$
N(x,y;x_0,y_0)\gg \eta_1\eta_2+O((\eta_1+s_1+\eta_2+s_2)\min\{\log(\eta_1+s_1),\log(\eta_2+s_2)\}). 
$$
Since $c_1\asymp c_2$, by \eqref{sbound}, \eqref{eta_i} and \eqref{range}, we have that $\eta_i+s_i\ll \eta_i$. Then, using \eqref{eta_i}, we get that
$$
N(x,y;x_0,y_0)\gg \eta_1\eta_2\gg \frac{c_1^2}{\lambda_1\lambda_2}
$$
provided that $c_1$ is sufficiently large (depending on $n$). Substituting this into   \eqref{eq4.15} and using 
\eqref{det(g)} together with \eqref{eta_i} we obtain that
\begin{align}
\nonumber\#\{\bv\in B\cap\Lambda:\bv+\bs\text{ is primitive}\}&\gg \frac{c_1^{n+1}}{\lambda_1\cdots\lambda_{n+1}}\asymp \frac{c_1^{n+1}}{\det(\Lambda)}\\ 
&\asymp_{c_1} \ep^{m+d/2} e^{(d/2+1)t}=\ep^m e^{t(d+1)} \left(\frac{\ep}{e^t}\right)^{d/2},\label{eq4.16}
\end{align}
and this completes the proof (note that $c_1$ depends on $n$ only).
\end{proof}

Combining Lemma~\ref{lem: mostly good} and Proposition~\ref{prop: count on major} gives the following theorem, which in turn verifies 
Theorem~\ref{thm: main} since any non-degenerate manifold $\cM$ can be parametrised locally by a map $F$ as above.

\begin{theorem}\label{thm: main for F}
Let $F$ be an analytic map of the form \eqref{eq: F},  nondegenerate at $\bx_0\in\bU$. Then there exists a ball $B_0$ centred at $\bx_0$ and constant $c=c(B_0, F, n, d)>1$ and a sufficiently large $t_0=t_0(n,F,B_0)$ such that for any $t\geq t_0$ and any $\ep\in(0,1)$ satisfying \eqref{lower bound on ep} 
$$
N(B_0; \ep, t)\gg_{n, F}  \ep^m e^{t(d+1)} \mu_d(B_0). 
$$
\end{theorem}

\begin{remark}\label{rem: count for more general}
   We used analyticity only in Lemma~\ref{lem: minor is small}. Thus in the view of Remark~\ref{rem: just nondeg}, we get the same conclusion as in Theorem~\ref{thm: main for F} in the range $\ep\gg e^{-3t/(2n-1)}$ for all non-nondegenerate manifolds, without assuming analyticity or even $C^\infty$. This itself yields a new result extending the results in \cite{SST} to all nondegenerate manifolds.
\end{remark}

\section{Further improvements}

In this section, we point towards how our method also yields an even better exponent than $\frac{3}{2m+1}$ in Theorem~\ref{thm: main} by imposing more geometric conditions beyond nondegeneracy. We begin by defining this condition. In what follows, as before, $\tilde F=(1,F)$.

\begin{definition}[$k$-high nondegeneracy]
Let $0\leq k\leq d-1$ be an integer. An analytic map $F:\bU\subset \R^d\to \R^n$ will be called \emph{$k$-high nondegenerate} at $\bx_0\in \bU$, if there exists a neighborhood $V$ of $\bx_0$ such that for any ball $B\subset V$, for any linearly independent vectors $\ba_1,\dots, \ba_{k+2}\in \R^{n+1}$
\begin{equation}\label{eq5.1}
\max_{\substack{I\subset\{1,\dots,d\}\\ \#I=k+1}}\; \sup_{\bx\in B} \left\vert \tilde{F}(\bx)\wedge\bigwedge_{i\in I} \partial_{i}\tilde{F}(\bx)\cdot (\ba_1\wedge\dots\wedge\ba_{k+2})\right\vert\neq 0,
\end{equation}
$F$ will be called \emph{$k$-high nondegenerate} if $F$ is $k$-high nondegenerate at almost every $\bx_0$.
\end{definition}
The above condition means that on a dense set of points $\bx$ in $V$, any $n-k$ dimensional subspace of $\R^{n+1}$ will have only trivial intersection with the subspace generated by $\tilde F$ and some (suitably chosen) $k+1$ directional derivatives. 

\begin{remark} The usual definition of nondegeneracy of $F$ stated in section~\ref{sec: notation} implies $0$-high nondegeneracy by Lemma~\ref{lem:crucial nabla} and the fact that 
$$
\Vert \tilde\nabla (\tilde F(\bx)\cdot\ba, \tilde F(\bx)\cdot\mathbf{b})\Vert=\max_{i=1}^d \Vert (\tilde F(\bx)\wedge \partial_{i}\tilde F(\bx))\cdot (\ba\wedge \mathbf{b})\Vert.
$$
\end{remark}
\begin{remark}
To understand the above definition in the case of the largest $k=d-1$, let $T_{\bx}$ be $d+1$ dimensional subspace generated by $\tilde F(\bx), \partial_1\tilde F(\bx), \dots, \partial_{d}\tilde F(\bx)$, i.e. tangent space to the cone $$\{(r, rF(\bx))~|~r>0, \bx\in \bU\}.$$ Then we can restate the $(d-1)$-high nondegeneracy at $\bx_0$ as the existence of a neighbourhood $V\ni\bx_0$ such that for any ball $B\subset V$, for any $(n-d)$-dimensional subspace $W\subset\R^{n+1}$
$$
\exists ~\bx\in B\quad T_{\bx}\cap W=\{\mathbf{0}\}.$$ 
\end{remark}
\begin{remark}
    Now, we give another interpretation of $d-1$-high nondegeneracy condition. Let us consider the following subset for a fixed subspace $W$ in $\R^{n+1}$ we define the following Schubert variety:
$$\Sigma_{W}=\{V\in \mathrm{Gr}(d+1, n+1)~|~ \dim(W\cap V)\ge 1\}.$$ Then our condition is equivalent to saying that for any ball $B$ 
$$
\exists\;\bx\in B \qquad T_{\bx}\notin \Sigma_{W}~ \forall~ W\in \mathrm{Gr}(m, n+1).$$

\end{remark}
\begin{remark}
Note that high nondegeneracy as introduced above is weaker than the \textit{strong curvature} conditions imposed in the works of \cite{Sc-Yam} and many subsequent papers, as these strong curvature conditions put constraints on the dimension of the manifold, whereas our condition is generic in the sense that no dimension restrictions are implied.
\end{remark}

\begin{remark}
We note down some examples which do not satisfy the high nondegeneracy. Let us consider $F(x,y)=(x,y, y^2,y^3)$ in $\R^4$; then it is not hard to see that $F$ is nondegenerate but not $1$-high nondegenerate. The corresponding manifold will contain a (rational) line.  
For a more subtle example of a non-degenerate but not $1$-high nondegenerate, take $$
F(x,y)=(x,y, (x^\pi+y^2), y^3),
$$
then for every $(x,y)$
$$
T_{x,y} \cap \big(\R(0,1,0,0,0)+\R(0,0,0,1,0)\big)\neq\{\mathbf{0}\}.$$ 
\end{remark}

The method we developed in the main part of the paper yields a better exponent by imposing the condition of high nondegeneracy. The corresponding result is as follows.

\begin{theorem}\label{thm:new main thm}
   Suppose  $F:\bU\to \R^n$ is $k$-high nondegenerate at $\bx_0\in \bU$ for some $0\leq k\leq d-1,$ and analytic then 
$$\eta(\cM)\geq \frac{3+k}{2m+1+k}.$$
\end{theorem}
For $k=d-1$, this gives
$ 
\eta(\mathcal{M})\geq \frac{d+2}{2m+d}$, which is the best one can get with this approach for all $(d-1)$-high nondegenerate manifolds and by plugging $k=0$, we get 
$\eta(\mathcal{M})\geq \frac{3}{2m+1}$ for all nondegenerate manifolds.

The proof of the above is the same as that of  Theorem~\ref{thm: main} once Proposition~\ref{prop: crucial} is replaced with the following modification, which is crucial for the improvement.

\begin{proposition}\label{prop: crucial_new}
    Let $F:\bU\to\R^n$ be analytic and $k$-high nondegenerate at $\bx_0\in\bU$ for some $0\leq k\leq d-1. $
Then there exists a neighbourhood $V$ of $\bx_0$ and $\alpha>0$ depending on $F$ and $\bx_0$ such that for any ball $B\subset V$ there exists $E>0$ such that for any $b=\diag(b_0, b_1,\dots, b_n)\in \mathrm{GL}_{n+1}(\R)$ satisfying
\begin{equation}\label{eq: condition on b_new}
b_0\geq 1,\quad b_{d+1}=\dots=b_n\leq 1,\quad b_{1}=\dots=b_d\geq 1,
\end{equation}
and
\begin{equation}\label{b_0 is too large 2_new}
    b_0 b_1^{k+1} b_{d+1}^{m}\geq 1,
\end{equation}
for any $\ep>0$, 
\begin{equation}
    \mu_d\{\bx\in B~|~\lambda_1(bu_1^\star(\bx)\Z^{n+1})<\ep\}\leq E \ep^\alpha \mu_d(B).
\end{equation}

\end{proposition}
\begin{proof}
The proof is the same as that of Proposition~\ref{prop: crucial}
with some modification that we will point out here. For any primitive collection $\bv_1,\dots,\bv_r\in\Z^{n+1}$, let  $\bu_1,\dots,\bu_{n+1-r}\in\Z^{n+1}$ be the primitive collection
as in Lemma~\ref{lem: dual prim}.
Take any multi-index of the form
$
I=\{0,i_1,\dots, i_{s+k+1}, j_1, \dots, j_{r-2-s-k}\},
$
where $
0\le s\le r-2-k,\; 
\{j_1,\dots, j_{r-2-s-k}\}\subset \{d+1,\dots, n\},\; \{i_1,\dots,i_{s+1+k}\}\subset \{1,\dots, d\}.$
Then, 
\begin{equation}\label{eq: covolume_new}
\begin{aligned}
    &\vert (h(\bx)\bv_1\wedge\dots\wedge h(\bx)\bv_r)_I\vert\\
    &=b_0 \prod_{p=1}^{s+1+k} b_{i_p}\prod_{p=1}^{r-2-s-k}b_{j_p}\vert \tilde F(\bx)\wedge_{p=1}^{s+1+k} \partial_{i_p}\tilde F(\bx)\wedge_{p=1}^{r-2-s-k} \be_{j_p}\wedge\bu_1\wedge \dots\wedge \bu_{n+1-r}\vert. 
\end{aligned}
\end{equation}

Let $s$ be the smallest integer in the range $0\le s\le r-2-k$ such that there exists a $\{j_1,\dots,j_{r-2-s-k}\}\subset \{d+1,\dots, n\}$ satisfying
\begin{equation}\label{eq: wedge nonzero_new}
\wedge_{p=1}^{r-2-s-k} \be_{j_p}\wedge\bu_1\wedge \dots\wedge \bu_{n+1-r}\neq \mathbf{0}.\end{equation} 
We will consider two different scenarios for the value of $s$.

\medskip

\noindent\textbf{Case $1$:} $s=0$. Note that the $s=0$ case can only happen if $r\leq m+2+k$. In this case, by Lemma~\ref{lem: dual prim}, there exists $\ba
_1,\cdots, \ba_{k+2}\in\Z^{n+1}$ orthogonal to every vector appearing in \eqref{eq: wedge nonzero_new}. Then, using Lemma~\ref{lem: dual basis consq} we get that for any $\{i_1,\cdots, i_{k+1}\}\subset \{1,\dots, d\}$
$$
\begin{aligned}
&\vert \tilde F(\bx)\bigwedge_{p=1}^{k+1} \partial_{i_p}\tilde F(\bx)\wedge_{p=1}^{r-2-k} \be_{j_p}\wedge\bu_1\wedge \dots\wedge \bu_{n+1-r}\vert 
 =\vert (\tilde F(\bx)\bigwedge_{p=1}^{k+1} \partial_{i_p}\tilde F(\bx))\cdot (\ba_1\wedge\cdots\wedge \ba_{k+2})\vert.
\end{aligned}
$$

\noindent By the definition of $k$-high nondegeneracy, there exists $\bx\in B$ and $\{i_1,\cdots, i_{k+1}\}\subset \{1,\dots, d\}$ such that for some $\rho_2>0$ independent from $\bv_1,\dots,\bv_r$
$$\vert (\tilde F(\bx)\bigwedge_{p=1}^{k+1} \partial_{i_p}\tilde F(\bx))\cdot (\ba_1\wedge\cdots\wedge \ba_{k+2})\vert>0.$$
Then for  $I=\{0, i_1, \cdots, i_{k+1}, j_1,\dots, j_{r-2-k}\}$ we obtain, following the same lines of argument as in the proof of Proposition~\ref{prop: crucial}, that
\begin{equation*}
    \sup_{\bx\in B}\Vert h(\bx)(\bv_1\wedge\dots\wedge\bv_r)\Vert\geq \vert h(\bx_1)(\bv_1\wedge\dots\wedge\bv_r)\vert_{I}\stackrel{\eqref{eq: condition on b_new}}{>} b_0 b_1^{k+1} b_{d+1}^{m}\rho_2/2\stackrel{\eqref{b_0 is too large 2_new}}{>}\rho_2/2>0.
\end{equation*}
since $b_{d+1}\leq 1.$

\medskip

\noindent\textbf{Case $2$:}
The proof for this case is the same as in Case $2$ of the proof of Proposition~\ref{prop: crucial}, except that $r-2$ is replaced by $r-2-k$. Then the same calculation yields,
 \begin{equation}
     \sup_{\bx\in B}\Vert h(\bx)(\bv_1\wedge\dots\wedge\bv_r)\Vert>b_0 b_1^{s+1+k} b_{d+1}^{r-s-2-k}\rho_1> b_0b_1^{k+1} b_{d+1}^{m} b_1^{s} >\rho_1>0.
 \end{equation}
Here, we used the fact that $r-s-2-k\leq m$, and that $b_{d+1}\leq 1$. 
Thus both of the cases combined yield \eqref{want: lower bound on covol}, and the proposition follows from Theorem KM.
\end{proof}

Thus using above proposition we get the following modified version of Theorem~\ref{new thm: BKM}. \begin{theorem}\label{new thm: BKM_new}
Let $F,\bU, m,d,n$ be as in Theorem~\ref{new thm: BKM}. Suppose further that $F$ is $k$-high nondegenerate at $\bx_0$ for some $\bx_0\in \bU$. Then there exists an open neighborhood $V\ni\bx_0$ and $\alpha>0$ with the following property. For any ball $B\subset V$, there exists a constant $E$ such that for any choice of $\delta>0$, $K>0$, and $T>0$ satisfying
\begin{equation}\label{condition on K T2}
 \max\{\delta^{m+2+k}, K^{m+k+2}\}\leq \delta K^{k+1} T^m,\qquad \delta K^{k+1}T^m\leq T^{m+2+k},
\end{equation}
we have that
\begin{equation}\label{eq: crucial improvement2}
\mu_d(\Deltanew(\delta, K, T)\cap B)\leq E (\delta K^{k+1}T^m)^{\alpha/(m+k+2)}\mu_d(B)\,.
\end{equation}
\end{theorem}

\begin{proof}

We will apply Proposition~\ref{prop: crucial_new}, with 
$$\begin{aligned}
&b_0=\delta^{-1}(\delta K^{k+1}T^m)^{1/(m+k+2)},\\ &  b_i= K^{-1}(\delta K^{k+1} T^m)^{1/(m+k+2)}, \quad 1\leq i\leq d,\\ & b_{j}=T^{-1}(\delta K^{k+1} T^m)^{1/(m+k+2)}, \quad d+1\leq j\leq n+1.
\end{aligned}
$$ 

Conditions \eqref{eq: condition on b_new} and \eqref{b_0 is too large 2_new} imposed on $b$ within Proposition~\ref{prop: crucial_new} are satisfied. Using the form of $F$ given by \eqref{eq: F} and the definitions of $b$ and $u_1^*$ we observe that
$$
\Deltanew(\delta, K, T)\subset \{\bx\in \bU~|~\lambda_1(bu_1^\star(\bx)\Z^{n+1})\leq (\delta K^{k+1} T^m)^{1/{m+k+2}}\}.
$$ 
\end{proof}

As an application we get the following generalisation of Lemma~\ref{lem: minor is small}.

\begin{lemma}\label{lem: minor is small_new}
Let $F$ be an analytic map of the form \eqref{eq: F}, $k$-high nondegenerate at $\bx_0\in\bU$. Let $\kappa>0$. Then there exists a ball $B_0$ centred at $\bx_0$ and constant $c=c(B_0, F, n, d,\kappa)>1$ such that for any $t\geq 1$ and any $\ep\in(0,1)$ satisfying \begin{equation}\label{lower bound on ep_new}
    \ep> c\,e^{-\frac{3t}{2m+1}},
\end{equation}
we have that
$$\mu_d(\mathfrak{M}(\varepsilon,t)\cap B_0)\leq  \kappa\mu_d(B_0).
$$ Moreover, $\mathfrak{M}(\varepsilon,t)$ can be written as a union of balls of radius $\ep^{1/2}e^{-t/2}$ of intersection multiplicity at most $N_d$, the Besicovitch constant.
\end{lemma}

The proof of the above is exactly the same as in Lemma~\ref{lem: minor is small}; we get a better range as in \eqref{lower bound on ep_new} due to a better measure estimate in Theorem~\ref{new thm: BKM_new}. The proof of Theorem~\ref{thm:new main thm} follows now from Lemma~\ref{lem: minor is small_new} and Proposition~\ref{prop: count on major}.

\subsubsection*{Acknowledgment}
SD thanks Nicolas de Saxcé for pointing to her that conditions appearing in Section 5 can be written in terms of avoiding certain varieties in the Grassmannian. The authors were supported by an EPSRC grant EP/Y016769/1. SD was also supported by a \textit{start-up} grant provided by the Indian Statistical Institute. She thanks the mathematical research institute MATRIX in Australia, where part of this research was presented. Also, she must thank Ben, Gaurav and Manuel for making the great ocean road trip possible while all of them were at MATRIX.

\bibliographystyle{abbrv}
\bibliography{inhomosim}

\end{document}